\documentclass[12pt]{amsart}
\usepackage{amsfonts,amssymb}

\theoremstyle{plain}
  \newtheorem{thm}{Theorem}[section]
  \newtheorem{prop}[thm]{Proposition}
  \newtheorem{cor}[thm]{Corollary}
  \newtheorem{lem}[thm]{Lemma}
  \newtheorem{claim}[thm]{Claim}
  \newtheorem{prob}[thm]{Problem}
  \newtheorem{quest}[thm]{Question}
  \newtheorem{conj}[thm]{Conjecture}

\theoremstyle{definition}
  \newtheorem{dfn}[thm]{Definition}

\theoremstyle{remark}
  \newtheorem{rem}[thm]{Remark}

\newcounter{yon}
\numberwithin{equation}{section}

	\newcommand{\dist}{\mathop{\mathit{d}} \nolimits}
	
	\newcommand{\diam}{\mathop{\mathrm{diam}} \nolimits}
	\newcommand{\ric}{\mathop{\mathit{Ric}}  \nolimits}
	\newcommand{\tra}{\mathop{\mathrm{Tra}} \nolimits}
	\newcommand{\mushi}{\mathop{\mathrm{def}} \nolimits}
	\newcommand{\pr}{\mathop{\mathrm{proj}} \nolimits}

	\newcommand{\di}{\mathop{\mathrm{di}}   \nolimits}

	\newcommand{\sep}{\mathop{\mathrm{Sep}} \nolimits}
	
	\newcommand{\obs}{\mathop{\mathrm{ObsDiam}}  \nolimits}
	
	\newcommand{\lip}{\mathop{\mathcal{L}ip}      \nolimits}
	\newcommand{\me}{\mathop{\mathrm{me}}      \nolimits}
	\newcommand{\supp}{\mathop{\mathrm{Supp}}    \nolimits}

	\newcommand{\obd}{\mathop{\underline{H}_{\lambda}\mathcal{L}\iota_1}    \nolimits}
	\newcommand{\obdd}{\mathop{\underline{H}_{1}\mathcal{L}\iota_1}
	\nolimits}
	\newcommand{\bobd}{\mathop{H_{\lambda}\mathcal{L}\iota_1}
	\nolimits}
    \newcommand{\bbobd}{\mathop{H_{1}\mathcal{L}\iota_1}
	\nolimits}

	\newcommand{\vol}{\mathop{\mathit{vol}}        \nolimits}

    \newcommand{\hess}{\mathop{\mathrm{Hess}}                \nolimits}
    \newcommand{\ent}{\mathop{\mathrm{Ent}}
	\nolimits}
    
    \newcommand{\e}{\mathop{\varepsilon}       \nolimits}

\title[Eigenvalues and isoperimetric
constants]{Eigenvalues of Laplacian and Multi-way isoperimetric
constants on weighted Riemannian manifolds}
\author{Kei Funano}

\thanks{\hspace{-0.4cm}Supported by a Grant-in-Aid for
Scientific Research from the Japan Society for the Promotion of Science.\\Department of Mathematics, Faculty of Science,\\
Kyoto University, Kyoto 606-8502, JAPAN\\
\textit{e-mail}: kfunano@math.kyoto-u.ac.jp}

\date{}
\begin{document}
\maketitle
\begin{abstract}We investigate the distribution of eigenvalues of the weighted
 Laplacian on closed weighted Riemannian manifolds of nonnegative Bakry-\'Emery Ricci curvature. We derive some universal inequalities among eigenvalues of the weighted Laplacian on such manifolds. These
 inequalities are quantitative versions of the previous theorem by the author with
 Shioya. We also study some geometric quantity, called multi-way
 isoperimetric constants, on such manifolds and obtain similar universal
 inequalities among them. Multi-way isoperimetric constants are generalizations
 of the Cheeger constant. Extending and following the heat semigroup argument by Ledoux and
 E.~Milman, we extend the Buser-Ledoux result to the $k$-th eigenvalue
 and the $k$-way
 isoperimetric constant. As a consequence the $k$-th eigenvalue of the
 weighted Laplacian and the $k$-way isoperimetric constant are equivalent up
 to polynomials of $k$ on closed weighted manifolds of nonnegative
 Bakry-\'Emery Ricci curvature. 
 
\end{abstract}
\section{Introduction}
\subsection{Eigenvalues of the weighted Laplacian}
Let $(M,\mu)$ be a pair of a Riemmanian manifold $M$
and a Borel probability measure $\mu$ on $M$ of
the form $d\mu=\exp(-\psi)d\vol_M$, $\psi\in C^2(M)$.
We call such a pair $(M,\mu)$ an \emph{weighted Riemannian manifold}.
We define the \emph{weighted Laplacian} (also called \emph{Witten Laplacian}) $\triangle_\mu$ by
\begin{align*}
  \triangle_\mu := \triangle -  \nabla \psi \cdot \nabla,
\end{align*}
where $\triangle$ is the usual positive Laplacian on $M$.
If $M$ is closed, then the spectrum of the weighted Laplacian
$\triangle_\mu$ is discrete, where $\triangle_\mu$ is considered
as a self-adjoint operator on $L^2(M,\mu)$. We denote its eigenvalues
with multiplicity by
\begin{align*}
  0=\lambda_0(M,\mu)<\lambda_1(M,\mu)
  \leq \lambda_2(M,\mu)\leq \cdots \leq \lambda_k(M,\mu) \leq \cdots .
\end{align*}

In this paper we study the following problem:
\begin{prob}\label{prob1}\upshape How do $\lambda_1(M,\mu), \lambda_2(M,\mu), \cdots,
 \lambda_k(M,\mu), \cdots$ lie on the real line?
 \end{prob}
The above problem amounts to finding the relation among
  eigenvalues of the weighted Laplacian.

  In order to tackle Problem \ref{prob1} let us focus on diameter estimates
  in terms of eigenvalues of the weighted Laplacian due to Li and Yau
  \cite[Theorem 10]{liyau} and Cheng \cite[Corollary 2.2]{cheng} (see also \cite{setti}). Combining
  their results one could obtain that $\lambda_k(M,\mu)\leq
  c(k,n)\lambda_1(M,\mu)$ for any natural number $k$ and any closed
  weighted Riemannian manifold $(M,\mu)$ of nonnegative Bakry-\'Emery Ricci curvature, here $c(k,n)$ is
  a constant depending only on $k$ and the dimension $n$ of $M$. The dependence of
  the constant $c(k,n)$ on $n$ comes from Cheng's result. In order to bypass the dimension dependence of the inequality by Cheng,
we consider the \emph{observable diameter} $\obs ((M,\mu);-\kappa)$, $\kappa>0$, introduced
by Gromov in \cite{gromov}. The observable diameter comes from the study of 'concentration of measure phenomenon' and it might be
interpreted as a substitute of the usual diameter. See Definition
\ref{obdidef}. The observable diameter is closely related with the first nontrivial eigenvalue of the weighted Laplacian as was firstly
  observed by Gromov and V.~Milman in \cite{milgro}:
  \begin{align}\label{grmils1}
   \obs_{\mathbb{R}}((M,\mu);-\kappa) \leq \frac{6}{\sqrt{\lambda_1(M,\mu)}}\log \frac{1}{\kappa}.
        \end{align}
  Under assuming the nonnegativity of Bakry-\'Emery Ricci
       curvature, E.~Milman obtained the opposite inequality
       (\cite{emil2,emil3,emil1}):
       \begin{align*}\obs_{\mathbb{R}}((M,\mu);-\kappa) \geq
       \frac{1-2\kappa}{2\sqrt{\lambda_1(M,\mu)}}.
        \end{align*}See (\ref{2.2sss1}), Proposition \ref{2.2t5}, and Lemma \ref{obsep}
        for the proof of the above two inequalities. Observe that these two inequalities are independent
        of the dimension. One might regard the Gromov-V.~Milman inequality
  as a dimension-free Cheng's inequality for $k=1$ and also the
        E.~Milman inequality as a dimension-free Li-Yau's inequality.

  One of the main results in this paper is the following:
\begin{thm}\label{Mthm}There exists a universal numeric constant $c>0$ such that if
 $(M,\mu)$ is a closed weighted Riemannian manifold of nonnegative
 Bakry-\'Emery Ricci curvature and $k$ is a natural number, then we have
 \begin{align*}\lambda_k(M,\mu)\leq \exp (c  k)\lambda_1(M,\mu).
 \end{align*}
 \end{thm}
Theorem \ref{Mthm} also holds for a convex domain with $C^2$ boundary in a
closed weighted Riemannian manifold of nonnegative Bakry-\'Emery Ricci curvature and
with the Neumann boundary condition, the proof of which is identical.

 The crucial point of Theorem \ref{Mthm} is that the constant $\exp(ck)$
 is independent of the dimension and quantitative. In \cite[Theorem
 1.1]{funa6} the author proved with Shioya that the fraction $\lambda_k(M,\mu)/\lambda_1(M,\mu)$
 is bounded from above by some universal constant depending only on
 $k$. However the estimate was not quantitative since the proof in
 \cite{funa6} relies on some
 compactness argument.

 In Theorem \ref{Mthm}, the nonnegativity of Bakry-\'Emery Ricci
 curvature is necessary as was remarked in \cite{funa6}. In fact for any
 $\e>0$ there exists a closed Riemannian manifold $M$ of Ricci
 curvature $\geq -\e$ such that $\lambda_2(M)/\lambda_1(M)\geq
 1/\e$. Taking an appropriate scaling, some 'dumbbell space' becomes such
 an example, see \cite[Example 4.9]{funa6} for details. 

The following corollary
  corresponds to Cheng's inequality for general $k$:

 \begin{cor}\label{obdicor}There exists a universal numeric constant $c>0$ such that if
 $(M,\mu)$ is a closed weighted Riemannian manifold of nonnegative
 Bakry-\'Emery Ricci curvature and $k$ is a natural number, then we have
  \begin{align*}
   \obs_{\mathbb{R}}((M,\mu);-\kappa)\leq
   \frac{\exp(ck)}{\sqrt{\lambda_k(M,\mu)}}\log \frac{1}{\kappa}.
   \end{align*}
  \end{cor}

  Corollary \ref{obdicor} follows from Theorem \ref{Mthm} together with
  the Gromov-V.~Milman inequality (\ref{grmils1}).

  In order to treat the $k$-th eigenvalue we will
  work on the notion of \emph{'separation'}, which is regarded as a
  generalization of the concentration of measure phenomenon (see
  Subsection \ref{sepsec}). It tells the information whether there exists a pair or not which
  are not separated in some sense among any $k+1$-tuple subsets with a
  fixed volume. According to the work of Chung, Grigor'yan, and Yau \cite{cgy1,cgy2},
  it is related with the information of general eigenvalues of the
  weighted Laplacian (see Theorem \ref{2.2t4}). Under assuming the nonnegativity of Bakry-\'Emery Ricci
  curvature, we prove that if a nonseparated pair always exists among any
  $k+1$-tuple of subsets, then it also holds among any $k$-tuple of
  subsets. In order to prove it, we use the curvature-dimension
  condition CD$(0,\infty)$ in the sense of Lott-Villani \cite{lott-villani} and
  Sturm \cite{sturm3,sturm}, which is equivalent to the nonnegativity of
  Bakry-\'Emery Ricci curvature. Idea of the proof of Theorem \ref{Mthm}
  will be discussed in Section \ref{prfmain} in more detail. 
  \subsection{Multi-way isoperimetric constants}
   Let $(M,\mu)$ be a closed weighted Riemannian manifold. Recall that
 \emph{Minkowski's (exterior) boundary measure} of a Borel subset $A$ of $M$, which we denote by $\mu^+(A)$, is defined as 
 \begin{align*}
  \mu^{+}(A):=\liminf_{r\to 0} \frac{\mu(O_r(A))- \mu(A)}{r},
  \end{align*}where $O_r(A)$ denotes the open $r$-neighborhood of
  $A$. We consider the following geometric quantity:
 \begin{dfn}[{Multi-way isoperimetric constants}]\upshape For a natural number $k$, we define the \emph{$k$-way
  isoperimetric constant} as
  \begin{align*}
   h_k(M,\mu):= \inf_{A_0,A_1,\cdots,A_k} \max_{0\leq i\leq k}
   \frac{\mu^{+}( A_i)}{\mu(A_i)},
   \end{align*}where the infimum runs over all collections of $k+1$
  non-empty, disjoint Borel subsets $A_0,A_1,\cdots, A_k$ of $M$. $h_1(M,\mu)$
  is also called the \emph{Cheeger constant}.
  \end{dfn}
  Note that $h_k(M,\mu)\leq  h_{k+1}(M,\mu)$ by the
  definition. We are interested in the distribution of
  $h_1(M,\mu),h_2(M,\mu),\cdots, h_k(M,\mu),\cdots$ on the real line and
  the relation between $\lambda_k(M,\mu)$ and $h_k(M,\mu)$.

  The Cheeger-Maz'ja inequality
   (\cite{Mazya1,Mazya2,Mazya3} and \cite{cheeger}, see \cite[Theorem 1.1]{emil2}) states that
   \begin{align}\label{chemazs}
    h_1(M,\mu)\leq 2\sqrt{\lambda_1(M,\mu)}.
    \end{align}In \cite{lgt}, resolving a conjecture by Miclo
    \cite{miclo} (see also \cite{djm12}), Lee, Gharan, and Trevisan obtained a higher order
    Cheeger-Maz'ja inequality for general graphs. Although they proved
    it for graphs, by an appropriate
    modification of their proof (e.g., by replacing sums with
    integrals), it is also valid for weighted Riemannian manifolds. In
    Appendix, we will discuss a point that we have to be care when we treat their argument for the smooth setting. 
  \begin{thm}[{Lee et al. \cite[Theorem 3.8]{lgt}}]\label{leethm}There exists a universal
   numerical constant $c>0$ such that for all closed weighted Riemannian
   manifold $(M,\mu)$ and a natural number $k$ we have 
   \begin{align*}
    h_k(M,\mu)\leq c k^3 \sqrt{\lambda_k(M,\mu)}.
    \end{align*}
   \end{thm}
 Lee et al. proved the above order $k^3$ can be improved to $k^2$
 for graphs (\cite[Theorem 1.1]{lgt}). Since it is uncertain that Lemma $4.7$ in
 \cite{lgt} holds or does not hold for the case of weighted Riemannian
 manifolds, the author does not know $k^3$ order in Theorem \ref{leethm} can be improved to
 $k^2$ order.

 The opposite inequality of the Cheeger-Maz'ja inequality (\ref{chemazs}) was shown by Buser \cite{buser} and Ledoux
 \cite{led}. They proved the existence of universal numeric constant $c>0$
 such that 
 \begin{align}\label{busled}
  c\sqrt{\lambda_1(M,\mu)}\leq h_1(M,\mu)
  \end{align}for any closed weighted Riemannian manifold $(M,\mu)$ of
  nonnegative Bakry-\'Emery Ricci curvature. Combining Theorems \ref{Mthm}
  and \ref{leethm} with (\ref{busled}) we obtain
  \begin{align*}
   h_k(M,\mu)\lesssim k^3\sqrt{\lambda_k(M,\mu)}\lesssim k^3\exp
   (ck)\sqrt{\lambda_1(M,\mu)}\lesssim k^3\exp(ck)h_1(M,\mu),
   \end{align*}where $A \lesssim B$ denotes $A\leq C B$ for some
   universal numeric constant $C>0$. Consequently we have the following:

  \begin{thm}\label{Mthm2}There exists a universal numeric constant $c>0$ such that if
 $(M,\mu)$ is a closed weighted Riemannian manifold of nonnegative
 Bakry-\'Emery Ricci curvature and $k$ is a natural number, then we have
   \begin{align*}
    h_k(M,\mu)\leq k^3 \exp(c k)h_1(M,\mu).
    \end{align*}
   \end{thm}
   In \cite{mimura} Mimura obtained similar universal inequalities
 among multi-way isoperimetric constants for Cayley graphs.
 
 Following and extending the heat semigroup argument by Ledoux \cite{led}
 and E.~Milman \cite{emil1}, we obtain the extension of the Buser-Ledoux
 Theorem:

   \begin{thm}\label{extthm}Assume that a closed weighted Riemannian manifold
    $(M,\mu)$ has nonnegative Bakry-\'Emery Ricci curvature. Then for
    any natural number $k$ we have
    \begin{align*}
      (80k^3)^{-1}\sqrt{\lambda_k(M,\mu)} \leq h_k(M,\mu).
     \end{align*}
    \end{thm}

    As a consequence $h_k(M,\mu)$ and $\sqrt{\lambda_k(M,\mu)}$ are
    equivalent up to polynomials of $k$ under assuming the nonnegativity of
    Bakry-\'Emery Ricci curvature.

    \subsection{Application to the stability of eigenvalues of the weighted Laplacian and
  multi-way isoperimetric constants}In \cite[Section 5]{emil1}, E.~Milman
  obtained several stability results of the Cheeger constants on convex
  bodies. We apply Theorems \ref{Mthm} and \ref{Mthm2} to one of his
  theorem in \cite{emil1}.

   For a domain $\Omega $ with $C^2$ boundary in a complete Riemannian manifold,
   we denote by $\eta_k(\Omega)$ the $k$-th eigenvalue of Laplacian with
   Neumann condition.
 \begin{cor}Let $K$, $L$ be two bounded convex domains in
  $\mathbb{R}^n$ and assume that both $K$ and $L$ have $C^2$ boundary. If
  \begin{align*}
   \vol(K\cap L)\geq v_K \vol (K) \text{ and } \vol (K\cap L) \geq v_L
   \vol (L),
   \end{align*}then
  \begin{align*}
   \eta_k(K) \geq  \frac{\exp(-c k)v_K^4}{\{\log (1+1/v_L)\}^2}\eta_k(L),
   \end{align*}and
  \begin{align*}
   h_k(K)\geq \frac{\exp(-c k)v_K^2}{k^3 \log (1+1/v_L)}h_k(L).
   \end{align*}
  where $c>0$ is a universal numeric constant.
  \end{cor}In particular, if $\vol(K) \simeq \vol(L)\simeq \vol(K\cap L)$
  then $\eta_k(K)\simeq_k \eta_k(L)$ and $h_k(K)\simeq_k h_k(L)$. Here
 $A\simeq B$ (resp., $A\simeq_k B$) stands for $A$ and $B$ are
 equivalent up to universal numeric constants (resp., constants
 depending only on $k$).

 E.~Milman obtained the above corollary for $k=1$ (\cite[Theorem 1.7]{emil1}). The above
 corollary follows from his theorem together with Theorems \ref{Mthm}
 and \ref{Mthm2}.

In the same spirit we investigate the (rough) stability property of
eigenvalues of the weighted Laplacian and multi-way isoperimetric constants with respect to
perturbation of spaces (Section \ref{stabilitysection}). We discuss the case where two weighted
manifolds $M$ and $N$ of nonnegative Bakry-\'Emery Ricci curvature are close with respect to the concentration
topology introduced by Gromov in \cite{gromov}. Roughly speaking, the
two spaces $M$ and $N$ are close with respect to the concentration
topology if $1$-Lipschitz functions on $M$ are close to
those on $N$ in some sense.

 \subsection{Organization of the paper}
 Section 2 collects
 some back ground material. In
 Section 3, after explaining some basics of the theory of optimal transportation, we prove
 Theorem \ref{Mthm}. In Section 4, we prove Theorem \ref{extthm}. In
 Section 5 we study the (rough) stability property of eigenvalues of
 the weighted Laplacian and multi-way isoperimetric constants with
 respect to the concentration topology. In Section 6 we discuss several questions concerning this
 paper and some conjecture raised in \cite{funa6}. 
\section{Preliminaries}
We review some basics needed in the proof of the main theorems.
\subsection{Concentration of measure}\label{elck}


In this subsection we explain the known relation among the 1st
eigenvalue of the weighted Laplacian, the Cheeger constant, and the concentration
of measure in the sense of L\'evy and V.~Milman (\cite{levy}, \cite{mil2}).
 
Let $X$ be an \emph{mm-space}, i.e.,
a complete separable metric space with a Borel probability measure $\mu_X$.

\begin{dfn}[{Concentration function,~\cite{amil}}]\label{intdi}
  For $r>0$ we define the real number $\alpha_X(r)$ as the
  supremum of $\mu_X ( X\setminus  O_r(A) )$, where $A$ runs over all
  Borel subsets of $X$ such that $\mu_X(A)\geq 1/2$. The function
  $\alpha_X:(\,0,+\infty\,)\to \mathbb{R}$ is called
  the \emph{concentration function}.
\end{dfn}

\begin{lem}[{\cite{amil},~\cite[Lemma 1.1]{ledoux}}]\label{2.1l1}
  If $\mu_X(A)\geq \kappa>0$, then
  \begin{align*}
    \mu_X(X\setminus O_{r+r_0}(A))\leq \alpha_X(r)
  \end{align*}
  for any $r,r_0>0$ such that $\alpha_X(r_0)<\kappa$. 
\end{lem}



The following Gromov and V. Milman's theorem asserts that Poincar\'e
inequalities imply appropriate exponential concentration inequalities (\cite{milgro}, \cite[Theorem 3.1]{ledoux}).

\begin{thm}[{\cite{milgro}}]\label{2.1t1}
  Let $(M,\mu)$ be a closed weighted Riemannian manifold. Then we have
  \begin{align*}
    \alpha_{(M,\mu)}(r)\leq \exp(-\sqrt{\lambda_1(M,\mu)}r/3)
  \end{align*}
  for any $r>0$. In particular, we have
 \begin{align}\label{expche}
  \alpha_{(M,\mu)}(r)\leq \exp(-h_1(M,\mu)r/6)
  \end{align}for any $r>0$.
\end{thm}
The second statement (\ref{expche}) follows from the first statement together with the
Cheeger-Maz'ja inequality (\ref{chemazs}).

\begin{rem}\upshape Integrating
 Cheeger's linear isoperimetric inequality also implies the second
 inequality (\ref{expche}) (see \cite[Proposition 1.7]{milso}). 
 \end{rem}

In the series of works \cite{emil2,emil3,emil1}, E.~Milman obtained the
converse of Theorem \ref{2.1t1} under assuming the nonnegativity of Bakry-\'Emery
Ricci curvature. He proved that a uniform tail-decay of the
concentration function implies the linear isoperimetric inequality
(Cheeger's isoperimetric inequality) under assuming nonnegativity of Bakry-\'Emery
Ricci curvature. E.~Milman's theorem plays a key role in the
proof of the main theorems.

For an weighted Riemannian manifold
$(M,\mu)$, we define the (\emph{infinite-dimensional})
\emph{Bakry-\'Emery Ricci curvature tensor} as
\begin{align*}
  \ric_\mu := \ric_M + \hess \psi.
\end{align*}
\begin{thm}[{E.~Milman, \cite[Theorem 2.1]{emil3}}]\label{2.1t2}
  Let $(M,\mu)$ be a closed weighted Riemannian manifold of nonnegative
  Bakry-\'{E}mery Ricci curvature. If $\alpha_{(M,\mu)}(r)\leq \kappa$
 for some $r>0$ and $\kappa\in (\, 0,1/2\, )$, then
 \begin{align*}
  h_1(M,\mu)\geq \frac{1-2\kappa}{r}.
  \end{align*}In particular, we have
 \begin{align*}
  \lambda_1(M,\mu)\geq \Big(\frac{1-2\kappa}{2r}\Big)^2.
  \end{align*}
\end{thm}

 The key ingredient of E. Milman's approach to the above result is the
 concavity of isoperimetric profile under the assumption of the nonnegativity of
 Bakry-\'Emery Ricci
 curvature, the fact based on the regularity theory of isoperimetric
 minimizers (see \cite[Appendix]{emil2}). See also \cite{ledoux} for the
 heat semigroup approach to Theorem \ref{2.1t2}.


\subsection{Separation distance}\label{sepsec}We define the separation distance which
 plays an important role when treating eigenvalues of the weighted
 Laplacian. The separation distance was introduced by Gromov in \cite{gromov}. 

\begin{dfn}[Separation distance]\label{2.2d1}
  For any $\kappa_0,\kappa_1,\cdots,\kappa_k\geq 0$ with $k\geq 1$,
  we define the ($k$-)\emph{separation distance}
  $\sep(X;\kappa_0,\kappa_1, \cdots, \kappa_k)$ of $X$
  as the supremum of $\min_{i\neq j}\dist_X(A_i,A_j)$,
  where $A_0,A_1, \cdots, A_k$ are
  any Borel subsets of $X$ satisfying that $\mu_X(A_i)\geq \kappa_i$ for
  all $i=0,1,\cdots,k$.
\end{dfn}It is immediate from the definition that if $\kappa_i\geq
  \tilde{\kappa_i}$ for each $i=0,1,\cdots,k$, then
  \begin{align*}
   \sep(X;\kappa_0,\kappa_1,\cdots,\kappa_k)\leq \sep(X;\tilde{\kappa}_0,\tilde{\kappa}_1,\cdots,\tilde{\kappa}_k).
   \end{align*}Note that if the support of $\mu_X$ is connected, then
 \begin{align*}\sep(X;\kappa_0, \kappa_1, \cdots, \kappa_k)=0
  \end{align*}for any $\kappa_0,\kappa_1,\cdots, \kappa_k>0$ such that
  $\sum_{i=0}^k \kappa_i>1$.

   For a Borel subset $A$ of an mm-space $X$ we put
   \begin{align*}
    \mu_A:= \frac{\mu_X|_A}{\mu_X(A)}
    \end{align*}
   \begin{lem}\label{2.2l1}If $A$ satisfies $\mu_X(A)\geq \kappa$, then
    \begin{align*}
     \sep((A,\mu_A);\kappa_0,\kappa_1,\cdots,\kappa_k)\leq
     \sep(X;\kappa\kappa_0, \kappa\kappa_1,\cdots, \kappa\kappa_k)
     \end{align*}for any $\kappa_0,\kappa_1,\cdots, \kappa_k>0$.
    \begin{proof}Take $k+1$ Borel subsets $A_0,A_1,\cdots, A_k$ of $A$
     such that $\mu_A(A_i)\geq \kappa_i$ for any $i$. The lemma
     immediately follows from that $\mu_X(A_i)\geq \mu_X(A)\kappa_i\geq \kappa\kappa_i$.
     \end{proof}
    \end{lem}
    We denote the closed $r$-neighborhood of a subset $A$ in
    a metric space by $C_r(A)$.
   \begin{lem}\label{2.2l2}Let $X$ be an mm-space and put $r:= \sep (X, \kappa_0,\kappa_1,\cdots,
    \kappa_{k})$. Assume that $k$ Borel subsets $A_0,A_1,\cdots,
    A_{k-1}$ of $X$
    satisfy $\mu_X(A_i)\geq \kappa_i$ for every $i=0,1,\cdots, k-1$ and
    $\dist_X(A_i,A_j)>r$ for every $i\neq j$. Then
    we have
    \begin{align*}
     \mu_X\Big(\bigcup_{i=0}^{k-1} C_{r} (A_i) \Big)\geq 1-\kappa_k.
     \end{align*}
    \begin{proof}
     Suppose that for some $\e_0>0$,
     \begin{align*}
      \mu_X\Big(\bigcup_{i=0}^{k-1} C_{r+\e_0} (A_i) \Big)\leq 1-\kappa_k.
      \end{align*}Putting $A_k:=X\setminus \bigcup_{i=0}^{k-1}
     C_{r+\e_0} (A_i)$ we have $\mu_X(A_k)\geq \kappa_k$ and $\dist_X(A_k,A_i)\geq r+\e_0$ for any
     $i=0,1,\cdots,k-1$. Thus we get
     \begin{align*}
      r<\min_{i\neq j} \dist_X(A_i,A_j)\leq
      \sep(X;\kappa_0,\kappa_1,\cdots, \kappa_k)=r,
      \end{align*}which is a contradiction. Hence $
     \mu_X(\bigcup_{i=0}^{k-1} C_{r+\e} (A_i) )> 1-\kappa_k$
     for any $\e>0$. Letting $\e\to 0$ we obtain the conclusion.
     \end{proof}
    \end{lem}

    The following lemma asserts that exponential concentration
    inequalities and logarithmic 2-separation inequalities are equivalent:

\begin{lem}\label{2.2l3}Let $X$ be an mm-space. 
 \begin{enumerate}
            \item If $X$ satisfies
                  \begin{align}\label{2.2s1}
                   \sep (X;\kappa,\kappa)\leq \frac{1}{C}\log \frac{c}{\kappa}\end{align}for any $\kappa>0$,
                  then we have $\alpha_X(r)\leq c \exp (-C r)$ for
                  any $r>0$.
           \item Conversely, if $X$ satisfies $\alpha_X(r)\leq c'\exp(-C' r)$ for any
                 $r>0$, then we have
                 \begin{align*}
                  \sep (X;\kappa,\kappa )\leq \frac{2}{C'}\log \frac{c'}{\kappa}
                  \end{align*}for any $\kappa>0$.
 \end{enumerate}
 \begin{proof}(1) Assume that $X$ satisfies (\ref{2.2s1}) and let $A\subseteq
  X$ be a Borel subset such that $\mu_X(A)\geq 1/2$. For $r>0$ we put
  $\kappa:=\mu_X(X\setminus O_r(A))$. Since
  \begin{align*}
   r\leq \dist_X(X\setminus O_r(A), A)\leq \sep(X;\kappa,1/2)\leq
   \sep(X;\kappa,\kappa)\leq \frac{1}{C}\log \frac{c}{\kappa},
   \end{align*}we have $\kappa\leq c\log (-Cr)$, which gives the
  conclusion of (1).

  (2) Assuming that $\alpha_X(r)\leq c'\exp(-C'r)$, we take two Borel subsets
  $A,B\subseteq X$ such that $\mu_X(A)\geq \kappa$, $\mu_X(B)\geq
  \kappa$, and $\dist_X(A,B)=\sep(X;\kappa,\kappa)$. Let $\tilde{r}$ be
  any positive number satisfying
  \begin{align*}
  \alpha_X(\tilde{r})\leq c'\exp(-C'\tilde{r})<\kappa,
   \end{align*}i.e.,
  \begin{align*}
   \tilde{r}>\frac{1}{C'}\log \frac{c'}{\kappa}.
   \end{align*}Since $\mu_X(A)\geq \kappa$, by Lemma \ref{2.1l1}, we have
  \begin{align*}
   1-\mu_X(O_{2\tilde{r}}(A))\leq \alpha_X(\tilde{r})<\kappa.
   \end{align*}Hence we have
  \begin{align*}
   \mu_X(O_{2\tilde{r}}(A)\cap B)> (1-\kappa)+\kappa - 1=0,
   \end{align*}which yields $\sep(X;\kappa,\kappa)=\dist_X(A,B)\leq
  2\tilde{r}$. Letting $\tilde{r}\to C'^{-1}\log (c'/\kappa)$ we obtain (2).
  \end{proof}
 \end{lem}

Theorem \ref{2.1t1} together with Lemma \ref{2.2l3} (2) implies that for any closed
weighted Riemannian manifold $(M,\mu)$ we have 
\begin{align}\label{2.2sss1}
  \sep((M,\mu);\kappa,\kappa)
  \leq \frac{6}{\sqrt{\lambda_1(M,\mu)}}\log \frac{1}{\kappa}.
\end{align}

Chung, Grigor'yan, and Yau generalized the above inequality in the
following form:

\begin{thm}[{Chung et al. \cite[Theorem 3.1]{cgy2}}]\label{2.2t4}
  Let $(M,\mu)$ be a closed weighted Riemannian manifold. Then, for
  any $k\in \mathbb{N}$ and any $\kappa_0,\kappa_1,\cdots,\kappa_k>0$,
  we have
 \begin{align*}
  \sep((M,\mu);\kappa_0,\kappa_1,\cdots ,\kappa_k)\leq
  \frac{1}{\sqrt{\lambda_k(M,\mu)}}\max_{i \neq j}
  \log \Big(\frac{e}{\kappa_i\kappa_j}\Big).
  \end{align*}
\end{thm}
Combining Theorems \ref{leethm} and \ref{2.2t4} we obtain the following proposition:
   \begin{prop}\label{leeprop}
    There exists a universal numeric constant $c>0$ such that for all closed weighted Riemannian
   manifold $(M,\mu)$, a natural number $k$, and
    $\kappa_0,\kappa_1,\cdots, \kappa_k>0$, we have
    \begin{align*}
     \sep ((M,\mu);\kappa_0,\kappa_1,\cdots,\kappa_k)\leq \frac{c k^3}{h_k(M,\mu)}\max_{i\neq j}\log
     \frac{e}{\kappa_i \kappa_j}.
     \end{align*}
    \end{prop}

We end this subsection reformulating Theorem \ref{2.1t2} in terms of the
separation distance for later use. 

\begin{prop}\label{2.2t5}
  Let $(M,\mu)$ be a closed weighted Riemannian manifold
  of nonnegative Bakry-\'Emery Ricci curvature. Then, for any
 $\kappa>0$, we have
 \begin{align*}
    \sep ((M,\mu);\kappa,\kappa)\geq \frac{1-2\kappa}{h_1(M,\mu)}
  \end{align*}and
  \begin{align*}
    \sep ((M,\mu);\kappa,\kappa)\geq \frac{1-2\kappa}{2\sqrt{\lambda_1(M,\mu)}}.
  \end{align*}
 \begin{proof}We prove only the first assertion. The proof of the second
  assertion is identical to the first one. According to Theorem
  \ref{2.1t2}, if $r\in (\, 0,\infty\, )$ and
  $\kappa\in (\, 0,1/2 \, )$ satisfy
  \begin{align}\label{2.2s6}
   r<\frac{1-2\kappa}{h_1(M,\mu)},
   \end{align}then we have $\alpha_{(M,\mu)}(r)>\kappa$. There
  exists $A\subseteq M$ such that $\mu(A)\geq 1/2$ and $\mu(M\setminus
  O_r(A))> \kappa $. Hence we have
  \begin{align*}
  r=\dist_M(A,M\setminus O_r(A)) \leq \sep ((M,\mu);1/2,\kappa)\leq \sep ((M,\mu);\kappa,\kappa).
   \end{align*}Combining the above inequality with (\ref{2.2s6}) gives the conclusion.
  \end{proof}
\end{prop}

Since $\sep ((M,\mu);\kappa, \kappa)\leq \diam M$, the first inequality
of Theorem \ref{2.2t5} recovers the Li-Yau inequality \cite{liyau}.

\subsection{Three distances between probability measures}
Let $X$ be a complete separable metric space. We denote by $\mathcal{P}(X)$
the set of Borel probability measures on $X$.

 \begin{dfn}[Prohorov distance]\label{2.3d1}
   Given two measures $\mu,\nu \in \mathcal{P}(X)$ and $\lambda \geq 0$,
   we define the \emph{Prohorov distance} $\di_{\lambda}(\mu,\nu)$
   as the infimum of $\varepsilon >0$ such that
  \begin{align}\label{2.3s1}
   \mu ({C}_{\varepsilon}(A))\geq
    \nu(A)-\lambda \varepsilon \text{ and }\nu({C}_{\varepsilon}(A))\geq
    \mu(A)-\lambda \varepsilon
  \end{align}
  for any Borel subsets $A\subseteq X$.
\end{dfn}

For any $\lambda \geq 0$, the function $\di_{\lambda}$ is a complete
separable distance function on $\mathcal{P}(X)$. If $\lambda >0$, then
the topology on $\mathcal{P}(X)$ determined by the Prohorov distance
function $\di_{\lambda}$ coincides with that of the weak convergence
(see \cite[Section 6]{bil}).
The distance functions $\di_{\lambda}$ for all $\lambda > 0$
are equivalent to each other. Also it is known that if $\mu
({C}_{\varepsilon}(A))\geq \nu(A)-\lambda \varepsilon$ for any Borel
subsets $A$ of $X$, then $\di_{\lambda}(\mu,\nu)\leq \e$. In other
words, the second inequality in (\ref{2.3s1}) follows from the first one
(see \cite[Section 6]{bil}).

For $(x,y)\in X\times X$, we put $\pr_1(x,y):=x$ and
$\pr_2(x,y):=y$. For two finite Borel measures $\mu$ and $\nu$ on $X$,
we write $\mu \leq \nu$ if $\mu(A)\leq \nu(A)$ for any Borel subset
$A\subseteq X$. A finite Borel measure $\pi$ on $X\times X$ is called
a \emph{partial transportation} from $\mu \in \mathcal{P}(X)$ to $\nu
\in \mathcal{P}(X)$ if $(\pr_1)_{\ast}(\pi)\leq \mu$ and
$(\pr_2)_{\ast}(\pi)\leq \nu$. Note that we do not assume $\pi$ to be
a probability measure. For a partial transportation $\pi$ from $\mu$
to $\nu$, we define its \emph{deficiency} $\mushi \pi$ by $\mushi \pi
:= 1-\pi(X\times X)$. Given $\varepsilon >0$, the partial
transportation $\pi$ is called an \emph{$\varepsilon$-transportation}
from $\mu$ to $\nu$ if it is supported in the subset
\begin{align*}
  \{ (x,y)\in X\times X \mid \dist_{X}(x,y)\leq
  \varepsilon\}.
\end{align*}

\begin{dfn}[Transportation distance]\label{2.3d2}
  Let $\lambda\geq 0$. For two probability measures $\mu,\nu \in
  \mathcal{P}(X)$, we define the \emph{transportation distance}
  $\tra_{\lambda}(\mu,\nu)$ between $\mu$ and $\nu$
  as the infimum of $\varepsilon >0$ such
  that there exists an $\varepsilon$-transportation $\pi$ from $\mu$
  to $\nu$ satisfying $\mushi \pi \leq \lambda \varepsilon$.
\end{dfn}

The following theorem is due to V.~Strassen.

\begin{thm}[{\cite[Corollary 1.28]{villani2},
    \cite[Section $3\frac{1}{2}.10$]{gromov}}]\label{2.3t3}
  For any $\lambda>0$, we have
  \[
  \tra_{\lambda} = \di_{\lambda}.
  \]
\end{thm}

Let $(X,\dist_X)$ be a complete metric space. We indicate by
$\mathcal{P}^2(X)$ the set of all Borel probability measures $\nu \in
\mathcal{P}(X)$ such that
\[
\int_{X}\dist_X(x,y)^2 d\nu (y)<+\infty
\]
for some $x\in X$.

\begin{dfn}[($L^2$-)Wasserstein distance]\label{2.3d4}
  For two probability measures $\mu,\nu \in \mathcal{P}^2(X)$, we
  define the \emph{$L^2$-Wasserstein distance $\dist_2^W(\mu,\nu)$}
  between $\mu$ and $\nu$ as the infimum of
  \begin{align*}
    \Big(\int_{X\times X}\dist_X(x,y)^2  d\pi
    (x,y)\Big)^{1/2},
  \end{align*}
  where $\pi \in \mathcal{P}^2(X\times X)$ runs over all \emph{couplings}
  of $\mu$ and $\nu$,
  i.e., probability measures $\pi$ with the property that $\pi (A\times
  X)=\mu(A)$ and $\pi (X\times A)=\nu (A)$ for any Borel subset
  $A\subseteq X$. It is known that this infimum is achieved by some
 transport plan, which we call an \emph{optimal transport plan} for $\dist_2^W(\mu,\nu)$. 
\end{dfn}

If the underlying space $X$ is compact, then the topology on
$\mathcal{P}(X)$ induced from the $L^2$-Wasserstein distance function
coincides with that of the weak convergence (see \cite[Theorem 7.12]{villani2}).


\section{Proof of Theorem \ref{Mthm}}\label{prfmain}

       In order to prove Theorem \ref{Mthm} we need to explain some
       useful tools from the theory of optimal transportation. Refer to \cite{villani2,villani}
       for more details.

Let $(X,\dist_X)$ be a metric space. A rectifiable curve
$\gamma:[0,1]\to X$ is called a \emph{geodesic} if its arclength
coincides with the distance $\dist_X(\gamma(0),\gamma(1))$ and it has a
constant speed, i.e., parameterized proportionally to the arclength. We
say that a metric space is a \emph{geodesic space}
if any two points are joined by a geodesic between them. It is known
that $(\mathcal{P}^2(X),\dist_2^W)$ is compact geodesic space as soon as
$X$ is (\cite[Proposition 2.10]{sturm}). 


Let $M$ be a close Riemannian manifold. For two probability measures $\mu_0,\mu_1\in \mathcal{P}^2(M)$ which are
absolutely continuous with respect to $d \vol_M$, there is a unique geodesic $(
\mu_t)_{t\in [0,1]}$ between them with respect to the $L^2$-Wasserstein
distance function $\dist_2^W$ (\cite[Theorem 9]{mcc}).

 For an mm-space $X$ let us denote by $\Gamma$ the set of minimal
 geodesics $\gamma:[0,1]\to X$ endowed with the distance
 \begin{align*}
  \dist_{\Gamma}(\gamma_1,\gamma_2):=\sup_{t\in [0,1]}\dist_X(\gamma_1(t),\gamma_2(t)).
  \end{align*}Define the \emph{evaluation map} $e_t:\Gamma \to X$ for
  $t\in [0,1]$ as $e_t(\gamma):=\gamma(t)$. A probability measure
  $\Pi\in \mathcal{P}(\Gamma)$ is called a \emph{dynamical optimal
  transference plan} if the curve $\mu_t:=(e_t)_{\ast}\Pi$, $t\in [0,1]$, is a
  minimal geodesic in $(\mathcal{P}^2(X),\dist_2^W)$. Then
  $\pi:=(e_0\times e_1)_{\ast}\Pi$ is an optimal coupling of $\mu_0$ and
  $\mu_1$, where $e_0\times e_1:\Gamma\to X\times X$ is the ``endpoints''
  map, i.e., $(e_0\times e_1)(\gamma):=(e_0(\gamma),e_1(\gamma))$.

  \begin{lem}[{\cite[Proposition 2.10]{lott-villani}}]\label{3l1}If
  $(X,\dist_X)$ is locally compact, then any minimal geodesic
   $(\mu_t)_{t\in [0,1]}$ in
  $(\mathcal{P}^2(X),\dist_2^W)$ is associated with a dynamical optimal transference
  plan $\Pi$, i.e., $\mu_t=(e_t)_{\ast}\Pi$.
   \end{lem}

 Let $\mu$ and $\nu$ be two probability measures on a set $X$. We define
 the \emph{relative entropy} $\ent_{\mu}(\nu)$ of $\nu$ with respect to
 $\mu$ as follows. If $\nu$ is absolutely continuous with respect to $\mu$,
 writing $d\nu=\rho d\mu$, then
 \begin{align*}
  \ent_{\mu} (\nu):=\int_M \rho \log \rho   d \mu,
  \end{align*}otherwise $\ent_{\mu}(\nu):=\infty$.

  \begin{dfn}[{Curvature-dimension condition, \cite{lott-villani}, \cite{sturm3,sturm}}]\label{3d1}Let $K$ be a real number. We say that
   an mm-space satisfies the \emph{curvature-dimension condition} $CD(K,\infty)$
   if for any $\nu_0,\nu_1\in \mathcal{P}^2(X)$ there exists a minimal
   geodesic $( \nu_t )_{t\in [0,1]}$ in $(\mathcal{P}^2(X),\dist_2^W)$
   from $\nu_0$ to $\nu_1$ such that
   \begin{align*}
    \ent_{\mu_X}(\nu_t)\leq (1-t)\ent_{\mu_X}(\nu_0)+t
   \ent_{\mu_X}(\nu_1)-\frac{K}{2}(1-t)t \dist_2^W(\nu_0,\nu_1)^2
    \end{align*}for any $t\in [0,1]$.
   \end{dfn}

In the above definition, assume that both $\nu_0$ and $\nu_1$ are
absolutely continuous with respect to $\mu_X$. Then Jensen's inequality applied to the convex function $r\mapsto r\log r$ gives
\begin{align}\label{3s1}
     & \log \mu_X(\supp \nu_t)\\ \geq \ & -(1-t) \int_{M} \rho_0 \log
 \rho_0 d\mu_X -t\int_M  \rho_1 \log \rho_1 d\mu_X + \frac{Kt(1-t)}{2}
 \dist_2^W(\mu_0,\mu_1)^2, \notag
\end{align}where $\rho_0$ and $\rho_1$ are densities of $\nu_0$
 and $\nu_1$ with respect to $\mu_X$ respectively.
 In particular, for two Borel subsets $A,B\subseteq X$ with
 $\mu_X(A),\mu_X(B)>0$, we have
 \begin{align}\label{3s1cont}
  &\log \mu_X (\supp \nu_t)\\ \geq \ &(1-t)\log \mu_X(A)+t\log
  \mu_X(B)+\frac{Kt(1-t)}{2}\dist_2^W\Big(\frac{\mu_X|_A}{\mu_X(A)},
  \frac{\mu_X|_B}{\mu_X(B)} \Big)^2 \tag*{}
  \end{align}(\cite{sturm},\cite{ohta}).

 \begin{thm}[{\cite{cms1,cms2}, \cite{vrs}, \cite{sturm2}}]\label{3t2}For a complete weighted Riemannian manifold $(M,\mu)$, we
  have $\ric_{\mu}\geq K$ for some $K\in \mathbb{R}$ if and only if
  $(M,\mu)$ satisfies $CD(K,\infty)$.
  \end{thm}

Theorem \ref{Mthm} follows from the following key theorem together with
Theorem \ref{2.2t4} and Proposition \ref{2.2t5}.

\begin{thm}\label{3t2}
 Let $(M,\mu)$ be a closed weighted Riemannian manifold of nonnegative
 Bakry-\'Emery Ricci curvature. If $(M,\mu)$
 satisfies
 \begin{align}\label{3ss1}
  \sep((M,\mu);\underbrace{\kappa,\kappa, \cdots,\kappa}_{k+1\text{ times}})\leq
  \frac{1}{D} \log \frac{1}{\kappa^2}
  \end{align}for any $\kappa>0$, then we have
 \begin{align}\label{conc1}
  \sep((M,\mu);\underbrace{\kappa,\kappa, \cdots,\kappa}_{k\text{ times}})\leq
  \frac{c}{D} \log \frac{1}{\kappa^2}
  \end{align}for any $\kappa>0$ and for some universal numeric constant
 $c>0$.
 \end{thm}

 The idea of the proof of Theorem \ref{3t2} is the following. It turns out that it is enough to prove (\ref{conc1}) for sufficiently
 small $\kappa>0$ and sufficiently large $c>0$. We suppose the converse
 of this, i.e.,
 \begin{align*}
  \sep ((M,\mu);\underbrace{\kappa,\kappa, \cdots,\kappa}_{k\text{ times}})> \frac{c}{D}\log \frac{1}{\kappa^2}
  \end{align*}for sufficiently small $\kappa>0$ and sufficiently large
  $c>0$. Put $\alpha:= (c/D)\log (1/\kappa)$. By the definition of the
  separation distance there exists $k$
  Borel subsets $A_0,A_1.\cdots, A_{k-1}\subseteq M$ such that
  $\min_{i\neq j}\dist(A_i,A_j)>\alpha$ and $\mu(A_i)\geq \kappa$ for
  any $i$. If we choose the constant $c$ large enough so that
  \begin{align*}
   \sep ((M,\mu);\underbrace{\kappa,\kappa, \cdots,\kappa}_{k\text{
   times}}, \kappa^{100})\leq \sep
   ((M,\mu);\underbrace{\kappa^{100},\kappa^{100},
   \cdots,\kappa^{100}}_{k+1\text{ times}})\leq \alpha /100,
   \end{align*}then by Lemma \ref{2.2l2} we have
   \begin{align*}
    \mu \Big(\bigcup_{i=0}^{k-1} C_{\alpha/100}(A_i) \Big)\geq 1-\kappa^{100}.
    \end{align*}It means that if $\kappa>0$ is sufficiently small, the measure
    of the set
    $\bigcup_{i=0}^{k-1} C_{\alpha/100}(A_i)$ is nearly $1$. Although it is not true, we assume that
    \begin{align}\label{uso}
     \mu \Big(\bigcup_{i=0}^{k-1} C_{\alpha/100}(A_i) \Big)=1
     \end{align}in order to tell the idea of the proof. Putting $A:=C_{\alpha/100}(A_0)$ and
     $B:=\bigcup_{i=1}^{k-1} C_{\alpha/100}(A_i)$, we have $M=A\cup B$,
     $A\cap B=\emptyset$, $\mu(A)\geq \kappa,\mu(B)\geq \kappa$, and $\dist(A,B)\geq \alpha/2$. 

     Let $(\mu_t)_{t\in[0,1]}$ be a geodesic from
   $\mu_A:=(1/\mu(A))\mu|_A$ to $\mu$ with respect to $\dist_2^W$. For
   sufficiently small $t>0$ we have $\dist(x,A)<\alpha/2 \leq
   \dist(A,B)$ for any $x\in \supp \mu_t$, which gives $\supp
   \mu_t\subseteq A$. This leads a contradiction since by (\ref{3s1cont}) we have
   \begin{align*}
    \log \mu(A)\geq \log \mu(\supp \mu_t)\geq (1-t)\log \mu(A)+\log \mu(M),
    \end{align*}which implies $\log \mu(A)\geq 0$. Although (\ref{uso})
    is always not true, we show below that the above idea can be
    accomplished by controlling separated subsets and estimating
    average distances between them.

 \begin{proof}[Proof of Theorem \ref{3t2}]It suffices to prove that there exist two universal
  numeric constants $c_0,\kappa_0>0$ such that
  \begin{align}\label{3s2}
   \sep ((M,\mu);\underbrace{\kappa,\kappa, \cdots,\kappa}_{k\text{ times}})\leq \frac{c_0}{D}\log \frac{1}{\kappa^2}
   \end{align}for any $\kappa\leq \kappa_0$. In fact, if $\kappa\geq
  1/2$, then the left-hand side of the above inequality is zero and there is nothing to prove. In the
  case where $\kappa_0< \kappa \leq 1/2$, by (\ref{3s2}) we have 
  \begin{align*}
   \sep((M, \mu);\underbrace{\kappa,\kappa, \cdots,\kappa}_{k\text{
   times}})\leq \ &
   \sep((M,\mu);\underbrace{\kappa_0,\kappa_0,\cdots,\kappa_0}_{k\text{
   times}})\\
   \leq \ &\frac{c_0\log
   \frac{1}{\kappa_0^2}}{D\log \frac{1}{\kappa^2}} \log
   \frac{1}{\kappa^2}\\ \leq \ & \frac{c_0\log
   \frac{1}{\kappa_0^2}}{D\log 4} \log \frac{1}{\kappa^2},
   \end{align*}which implies the conclusion of the theorem.

  Suppose the contrary to (\ref{3s2}), i.e.,
  \begin{align}\label{3s3}
   \sep((M,\mu);\underbrace{\kappa,\kappa, \cdots,\kappa}_{k\text{ times}}) > \frac{c_1}{D} \log \frac{1}{\kappa^2},
   \end{align}where $c_1>0$ is a sufficiently large universal numeric
  constant and $\kappa>0$ is a sufficiently small number. Both the largeness of $c_1$ and the
  smallness of $\kappa$ will be specified later. Note that the assumption (\ref{3s3}) immediately gives $k
  \kappa<1$ (otherwise, the left-hand side of (\ref{3s3}) is zero). We denote
  the right-hand side of (\ref{3s3}) by $\alpha$, i.e.,
  \begin{align*}
   \alpha:=\frac{c_1}{D}\log \frac{1}{\kappa^2}.
   \end{align*}
  \begin{claim}\label{3cl1}If $c_1>0$ (resp., $\kappa>0$) in (\ref{3s3}) is large
   enough (resp., small enough), then there exist two closed subsets $B_0, B_1 \subseteq M$ such
   that $B_0\subseteq B_1$, $\kappa/4 \leq \mu(B_0)\leq 1/2$, $\mu(B_1)\geq
   1-\kappa^6$, and 
   \begin{align*}
    \dist_M(B_0,B_1\setminus B_0)\geq c_2 \max \Big\{
    \alpha, \frac{\kappa}{\sqrt{\lambda_1(M,\mu)}} \Big\}
    \end{align*}for some universal numeric constant $c_2>0$.
   \begin{proof}The assumption (\ref{3s3}) implies the existence of $k$ Borel
  subsets $A_{0}, A_{1}, \cdots , A_{k-1}\subseteq M$
  such that $\mu(A_{i})\geq \kappa$ for any $i$ and
\begin{align*}
    \dist_{M}(A_{i},A_{j})\geq \alpha \text{ for any }i\neq j.
 \end{align*}If $\kappa<1/8$ and $c_1\geq 8$, then by (\ref{3ss1}) we have 
\begin{align*}
 \sep ((M,\mu);\underbrace{\kappa,\kappa,
 \cdots,\kappa}_{k\text{ times}},1/8) \leq \sep
 ((M,\mu);\underbrace{\kappa,\kappa,  \cdots,\kappa}_{k+1\text{ times}}
 ) \leq \frac{1}{D}\log \frac{1}{\kappa^2}\leq \frac{\alpha}{8}.
 \end{align*}Hence Lemma \ref{2.2l2} yields
  \begin{align}\label{3s4}
   \mu\Big(\bigcup_{i=0}^{k-1} C_{\alpha/8 }(A_i)\Big)\geq \frac{7}{8}.
   \end{align}Note that
    \begin{align}\label{3s5}\dist_M(C_{\alpha/8}(A_i),
    C_{\alpha/8}(A_j))\geq \alpha/4
     \end{align}for any $i\neq j$. According to Proposition \ref{2.2t5} we take $X_0,X_1\subseteq M$ such that
  \begin{align}\label{3s6}
  \mu(X_i)\geq \frac{1}{2}-\frac{\kappa}{4} \ \  (i=0,1)
  \end{align}and
  \begin{align}\label{3s7}
   \dist_M(X_0,X_1)\geq \frac{\kappa}{8\sqrt{\lambda_1(M,\mu)}}.
   \end{align}Set $Y:=X_0\cup X_1$. By (\ref{3s6}) we have $\mu (Y)\geq
    1-\frac{\kappa}{2}$ and thus
  \begin{align*}
   \mu(Y\cap C_{\alpha/8}(A_i))\geq \mu(Y\cap A_i)\geq \Big(1-\frac{\kappa}{2}\Big)+\kappa -1 \geq
   \frac{\kappa}{2} 
   \end{align*}for each $i=0,1,\cdots, k-1$. Suppose that $\mu(X_i\cap
    C_{\alpha/8}(A_l))< \kappa/4$ for some $i\in \{0,1\}$ and for any
    $l=0,1,\cdots, k-1$. Then we have
  \begin{align*}
   \mu\Big(X_i \cap \bigcup_{l=0}^{k-1}C_{\alpha/8}(A_l)\Big)\leq \frac{k\kappa}{4}<\frac{1}{4}.
   \end{align*}Combining (\ref{3s4}) with (\ref{3s6}) we also get
  \begin{align*}
     \mu\Big(X_i \cap \bigcup_{l=0}^{k-1} C_{\alpha/8}(A_l)\Big)\geq
   \Big(\frac{1}{2}-\frac{\kappa}{4}\Big) +\frac{7}{8}- 1\geq \frac{1}{4},
   \end{align*}which is a contradiction (we have used $\kappa<1/8$). Therefore for each
    $l=0,1,\cdots, k-1$ we may choose $n_{l}\in \{0,1\}$ so that
    \begin{align}\label{3ss2}
     \mu(X_{n_l}\cap C_{\alpha/8}(A_l))\geq \kappa/4
     \end{align}and
    \begin{align*}
     I_i:=\{ l\in \{0,1,\cdots,k-1\} \mid n_l=i \}\neq \emptyset \ \ \ (i=0,1).
     \end{align*}For each $i=0,1$, we set
  \begin{align*}
   A_{i}':=X_i \cap \bigcup_{l\in I_{i}} C_{\alpha/8}(A_l). 
   \end{align*}Combining (\ref{3s5}) with (\ref{3s7}) yields
  \begin{align}\label{3s9}
   \dist_M(A_0',A_1')\geq \max \Big\{
   \frac{\kappa}{8\sqrt{\lambda_1(M,\mu)}},  \frac{\alpha}{4} \Big\}=:\beta.
   \end{align}Since
  \begin{align*}
   \sep\Big((M,\mu); \underbrace{\frac{\kappa}{2},\frac{\kappa}{2},
   \cdots,\frac{\kappa}{2}}_{k\text{ times}},
   \kappa^6\Big) \leq \sep ((M,\mu); \underbrace{\kappa^6,\kappa^6,
   \cdots,\kappa^6}_{k+1\text{ times}}) \leq
   \frac{c_3}{D}\log \frac{1}{\kappa^2}
   \end{align*}for some universal numeric constant $c_3>0$, we get
    \begin{align*}
      \sep\Big((M,\mu); \underbrace{\frac{\kappa}{2},\frac{\kappa}{2},
   \cdots,\frac{\kappa}{2}}_{k\text{ times}},
   \kappa^6\Big) \leq \frac{\alpha}{16}\Big(\leq \frac{\beta}{4}\Big)
     \end{align*}provided that $c_1$ in (\ref{3s3}) is large
    enough. Put $B_0:=C_{\alpha/16}(A_0')$ and
    $B_{1}:=C_{\alpha/16}(A_0')\cup C_{\alpha/16}(A_1')$. We may assume
    that $\mu(C_{\alpha/16}(A_0'))\leq 1/2$. Thanks to Lemma
    \ref{2.2l2} it is easy to check that $\mu(B_1)\geq 1-\kappa^6$. By
    (\ref{3ss2}) and (\ref{3s9}), we see that $B_0$ and $B_1$ possess
   the other desired properties.
    \end{proof}
   \end{claim}
We consider two Borel probability measures $\mu_{B_i}$, $i=0,1$, defined by
  \begin{align*}
   \mu_{B_i}:=\frac{\mu|_{B_i}}{\mu(B_i)}.
   \end{align*}

  The following claim is essentially due to Gromov \cite{gromov} (see
  also \cite[Claim 5.10]{funa6}). He used it in the
  context of the convergence theory of mm-spaces without detailed
  proof. Since our context is different from his one, we include the
  proof for the concreteness of this paper. The proof below is shorter than
  the one in \cite[Claim 5.10]{funa6}.

\begin{claim}[{\cite[Section $3\frac{1}{2}.47$]{gromov}}]\label{3cl2}
  There exist a universal numeric constant $c_4>0$ and a coupling $\pi$ of
 $\mu_{B_0}$ and $\mu_{B_1}$ such that
  \begin{align*}
   \pi \Big(\Big\{(x,y) \in M\times M \mid \dist_M(x,y)> \frac{c_4\log
   \frac{1}{\kappa^2}}{\sqrt{\lambda_1(M,\mu)}}\Big\}\Big)\leq \kappa^6.
   \end{align*}
 \begin{proof}We use the identity
  $\di_{\lambda}(\mu_{B_0},\mu_{B_1})=\tra_{\lambda}(\mu_{B_0},
  \mu_{B_1})$ (Theorem \ref{2.3t3}). Put
  $\delta:=\frac{c_4}{\sqrt{\lambda_1(M,\mu)}}\log \frac{1}{\kappa^2}$,
  where $c_4>0$ is a numeric universal constant which will be determined later. We shall prove that
  \begin{align}\label{3s10}
   \mu_{B_1}(C_{\delta}(A))\geq \mu_{B_0}(A)-\kappa^6
   \end{align}for any Borel subset
  $A\subseteq B_1$, which implies the claim. In fact, applying (\ref{3s10})
  to Theorem \ref{2.3t3} gives that there exists a $\delta$-transportation $\pi_0$ from
  $\mu_{B_0}$ to $\mu_{B_1}$ such that $\mushi \pi_0 \leq \kappa^6$. If $\mushi \pi_0 =0$, then we set $\pi:=\pi_0$. If $\mushi
  \pi_0 >0$, then set
  \begin{align*}
   \pi:= \pi_0+\frac{1}{\mushi \pi_0 } (\mu_{B_0} - (\pr_1)_{\ast} \pi_0) \times
   (\mu_{B_1}- (\pr_2)_{\ast}\pi_0 ).
   \end{align*}It is easy to check that $\pi$ fulfills the desired
  property. 

  To prove (\ref{3s10}) we may assume that $\mu_{B_0}(A)\geq \kappa^6$,
  which yields that
  \begin{align*}\mu_{B_1}(A)\geq \mu(A)\geq \kappa^6 \mu(B_0)\geq
   \kappa^7 / 4.
   \end{align*}Using Lemma \ref{2.2l1} and (\ref{2.2sss1}) we choose $c_4>0$ so that
  \begin{align*}
   \sep \Big((B_1,\mu_{B_1}); \frac{\kappa^7}{4},
   \frac{\kappa^7}{4}\Big)
   \leq \ &\sep \Big((M,\mu); (1-\kappa^6)\frac{\kappa^7}{4},
   (1-\kappa^6)\frac{\kappa^7}{4}\Big)\\ \leq \ &
   \frac{c_4}{\sqrt{\lambda_1(M,\mu)}}\log \frac{1}{\kappa^2} \ (=\delta).
   \end{align*}Lemma \ref{2.2l2} implies that
  \begin{align*}
  \mu_{B_1}(C_{\delta}(A))\geq 1- \frac{\kappa^7}{4}\geq 1-\kappa^6\geq \mu_{B_0}(A)-\kappa^6,
   \end{align*}which is (\ref{3s10}).
  \end{proof}
 \end{claim}We set
  \begin{align*}
   \Delta:= \Big\{ (x,y)\in M\times M  \mid
  \dist (x,y)\leq \frac{c_4\log \frac{1}{\kappa^2}}{\sqrt{\lambda_1(M,\mu)}}  \Big\}.
   \end{align*}
  We consider two Borel probability measures $\mu_{0}:=a
  (\pr_1)_{\ast}(\pi |_{\Delta})$ and $\mu_{1}:=a
  (\pr_2)_{\ast}(\pi |_{\Delta})$, where $a:=
  \pi(\Delta)^{-1}$. By Claim \ref{3cl2} we have
  \begin{align}\label{3s11}
   1\leq a \leq \frac{1}{1-\kappa^6}
   \end{align}and
  \begin{align}\label{3s12}
   \dist_2^W(\mu_{0},\mu_{1})^2 \leq a  \int_{M \times
   M }\dist (x,y)^2 d \pi |_{\Delta}(x,y) \leq \Big\{ \frac{c_4\log \frac{1}{\kappa^2}}{\sqrt{\lambda_1(M,\mu)}}\Big\}^2.
   \end{align}Take an optimal dynamical transference
  plan $\Pi$ such that $(e_i)_{\ast}\Pi =\mu_i$ for each
  $i=0,1$. Putting $r:=\dist_{M}(B_0,B_1\setminus
  B_0)$,
  we consider
  \begin{align*}
   \Gamma_t := \{   \gamma \in \supp \Pi  \mid \dist_{M}(e_0(\gamma),
   e_t(\gamma))\leq r/2                        \}.
   \end{align*}By (\ref{3s12}) we have
  \begin{align*}
   \frac{r^2}{4}\Pi (\Gamma\setminus \Gamma_t)\leq
   \dist_2^W((e_0)_{\ast}\Pi,(e_t)_{\ast}\Pi)^2= t^2\dist_2^W(\mu_0,\mu_1)^2 \leq \Big\{ \frac{c_4t\log \frac{1}{\kappa^2}}{\sqrt{\lambda_1(M,\mu)}}\Big\}^2.
   \end{align*}According to Claim \ref{3cl1} we thus get
  \begin{align}\label{3s13}
   \Pi(\Gamma_t)\geq 1- \frac{c_5 t^2 \Big(\log \frac{1}{\kappa^2}\Big)^2}{\kappa^2}
   \end{align}for some universal numeric constant $c_5>0$. For $s\in [0,1]$ we put $\nu_s := (e_s)_{\ast}
  \frac{\Pi|_{\Gamma_t}}{\Pi(\Gamma_t)}$. By the definition of $\nu_s$ we obtain the following. 
  \begin{claim}\label{3cl3} $\supp \nu_t
    \cap B_1 \subseteq B_0$.
   \end{claim}By using Claim \ref{3cl3}, we get
  \begin{align}\label{3s15}
   \log \mu (B_0) + \frac{\kappa^6}{\mu (B_0)}
   \geq \ & \log \mu (B_0)+ \log
   \Big(1+\frac{\kappa^6}{\mu (B_0)}\Big) \\
   = \ &\log (\mu (B_0) +\kappa^6) \tag*{}\\
   \geq \ &\log \{\mu (\supp \nu_t \cap B_1) +
   \mu (  \supp \nu_t      \setminus B_1)\}
   \tag*{}\\
   = \ & \log \mu (\supp \nu_t) \tag*{}
   \end{align}Note that $(\nu_s)_{s\in [0,1]}$ is a geodesic between
  $\nu_0$ and $\nu_1$. Since
  \begin{align}\label{3ss14} \nu_i= \frac{(e_i)_{\ast}\Pi|_{\Gamma_t}}{\Pi(\Gamma_t)} \leq
 \frac{(e_i)_{\ast}\Pi}{\Pi(\Gamma_t)} =  \frac{\mu_i}{\Pi(\Gamma_t)}\leq
   \frac{a}{\Pi(\Gamma_t)} (\pr_{i+1})_{\ast}\pi=
   \frac{a}{\Pi(\Gamma_t)}\mu_{B_i}
   \end{align}for $i=0,1$, each $\nu_i$ is absolutely continuous with
  respect to $\mu$, and especially the above geodesic $(\nu_s)_{s\in
  [0,1]}$ is unique. For each $i=0,1$, we write $d \nu_i = \rho_i d\mu$. By (\ref{3s1}), we get
  \begin{align}\label{3s14}
   &\log \mu (\supp \nu_t )  \geq   - (1- t)
   \int_{M} \rho_{0}\log \rho_{0} d\mu - t  \int_{M}
   \rho_{1} \log \rho_{1} d\mu.
   \end{align}For a subset $A\subseteq M$ we denote by $1_A$ the
  characteristic function of $A$, i.e., $1_A(x):=1$ if $x\in A$ and
  $1_A(x):=0$ if $x\in M\setminus A$.
\begin{claim}\label{3cl4}We have
 \begin{align*}
  \rho_{i}\log \rho_{i}\leq \frac{c_t
  1_{B_i}}{\mu (B_i)}\log  \frac{c_t
  1_{B_i}}{\mu(B_i)} \ \ (i=0,1), 
  \end{align*}where $c_t:= a/\Pi(\Gamma_t)$.
 \begin{proof}By (\ref{3ss14}) we have $\rho_{i}\leq (c_t
  /\mu(B_i))1_{B_i}$. Since
  $c_t\geq 1$ and $u \log u \leq v \log v$ for any two positive
  numbers $u,v$ such
  that $u\leq v$ and $v\geq 1$, we obtain the claim.
  \end{proof}
 \end{claim}Combining Claim \ref{3cl4} with (\ref{3s15}) and (\ref{3s14}) we have
  \begin{align*}
   & \log \mu (B_0) + \frac{\kappa^6}{\mu (B_0)}\\ \geq \ & -(1-t) \int_{M}\frac{c_t
  1_{B_0}}{\mu(B_0)}\log  \frac{c_t
  1_{B_0}}{\mu (B_0)} d\mu -t \int_{M} \frac{c_t
  1_{B_1}}{\mu (B_1)}\log  \frac{c_t
  1_{B_1}}{\mu (B_1)} d\mu  \\
   =\ & -c_t \log c_t + c_t (1-t)\log
   \mu(B_0)  + c_t t \log \mu(B_1).
   \end{align*}Substituting $t:=\kappa^3$, we thereby obtain
  \begin{align}\label{3s16}
   & \log (1/2) + 4\kappa^2\\ \geq 
   \ & \log \mu(B_0)+ \frac{\kappa^6}{\kappa^3 \mu(B_0)}\tag*{} \\ \geq \ & -
   \frac{c_t}{\kappa^3}\log c_t + \frac{c_t -1
   }{\kappa^3} (1-\kappa^3)\log \mu (B_0 ) + c_t \log \mu (B_1). \tag*{}
  \end{align}Using (\ref{3s11}) and (\ref{3s13}) we estimate each term on the
  right-side of the above inequalities as 
  \begin{align*}
   &\frac{c_t\log c_t}{\kappa^3}\\  =\ &
   \frac{a}{\Pi(\Gamma_t)}\cdot \frac{\log a - \log \Pi
   (\Gamma_t)}{\kappa^3}\\ \leq \ &\frac{1}{1-\kappa^6}  \Big(1-\frac{c_5\kappa^6 \Big(\log
   \frac{1}{\kappa^2}\Big)^2}{\kappa^2}\Big)^{-1}\\
   \ & \ \ \times \frac{1}{\kappa^3} \Big(\log \frac{1}{1-\kappa^6} - \log
   \Big(1-\frac{c_5\kappa^6 \Big(\log
   \frac{1}{\kappa^2}\Big)^2}{\kappa^2}\Big) \Big)\\ \leq \ &
   \frac{1}{1-\kappa^6} \Big(1-c_5 \kappa^4 \Big( \log
   \frac{1}{\kappa^2}\Big)^2\Big)^{-1}     \cdot 2 \Big(\kappa^3 + c_5\kappa \Big(\log \frac{1}{\kappa^2}\Big)^2\Big),
   \end{align*}
  \begin{align*}
       \Big|\frac{c_t -1 }{\kappa^3} \log \mu(B_0)\Big| \leq \ &\frac{a-\Pi(\Gamma_t)}{\kappa^3
   \Pi(\Gamma_t)}  \log \frac{2}{\kappa}\\ \leq \ &
   \frac{\frac{1}{1-\kappa^4}-1+c_5 \kappa^4\Big( \log
   \frac{1}{\kappa^2}\Big)^2}{\kappa^3 \Pi(\Gamma_t)}\log
   \frac{2}{\kappa}\\
   \leq \ &\kappa \frac{1+c_5(1-\kappa^4)\Big(\log
   \frac{1}{\kappa^2}\Big)^2}{(1-\kappa^4)\Big(1-c_5\kappa^4\Big(\log
   \frac{1}{\kappa^2}\Big)^2\Big)}\log \frac{2}{\kappa},
   \end{align*}and
  \begin{align*}
   |c_t\log \mu(B_1)|\leq \frac{a}{\Pi(\Gamma_t)}\log
   \frac{1}{1-\kappa^6}\leq \frac{2\kappa^6}{(1-\kappa^6)\Big(1-c_5\kappa^4\Big(\log\frac{1}{\kappa^2}\Big)\Big)}.
   \end{align*}These estimates imply the right-side of the
  inequalities (\ref{3s16}) is close to zero for sufficiently small
  $\kappa>0$. Since the left-side of the inequality (\ref{3s16}) is about $\log(1/2)<0$ for
  sufficiently small $\kappa>0$,
  this is a contradiction. This
  completes the proof of the theorem.
  \end{proof}

\section{Proof of Theorem \ref{extthm}}
On a closed weighted Riemannian manifold $(M,\mu)$, denote by
$(P_t)_{t\geq 0}$ the semigroup associated with the infinitesimal
generator $\Delta_{\mu}$. For each $t\geq 0$, $P_t:C^{\infty}(M)\to
C^{\infty}(M)$ is a bounded linear operator and we extend the action of $P_t$ to
$L^p(\mu)$ $(p\geq 1)$.

The following gradient estimate of the heat semigroup is due to
Bakry and Ledoux \cite{bakled}. One might regard it as a dimension-free
Li-Yau parabolic gradient inequality \cite{liyau2}.
\begin{lem}[{Bakry-Ledoux, \cite[Lemma 4.2]{bakled}}]\label{bakledlem}Let $(M,\mu)$ be a closed weighted
 Riemannian manifold of Bakry-\'Emery Ricci curvature bounded from below
 by a nonpositive real number $K$. Then for any $t\geq 0$ and $f\in C^{\infty}(M)$ we have
 \begin{align*}
  c(t)|\nabla P_t (f)|^2\leq P_t(f^2)-(P_t(f))^2,
  \end{align*}where
 \begin{align*}
  c(t):=\frac{1-\exp(2Kt)}{-K} \ (=2t \text{ if } K=0).
  \end{align*}
 \end{lem}
\begin{cor}\label{bakledcor}If $(M,\mu)$ has nonnegative Bakry-\'Emery Ricci curvature,
 then for any $t\geq 0$, $p\geq 2$, and $f\in C^{\infty}(M)$, we have
 \begin{align*}
  \| |\nabla P_t (f)| \|_{L^p(\mu)}\leq \frac{1}{\sqrt{2t}}\| f\|_{L^p(\mu)}.
  \end{align*}
 \end{cor}

 From Corollary \ref{bakledcor} Ledoux obtained the following lemma:
 \begin{lem}[{Ledoux, \cite[(5.5)]{led}}]\label{ledlem}Assume that $(M,\mu)$ has
  nonnegative Bakry-\'Emery Ricci 
  curvature. Then for any $f\in
  C^{\infty}(M)$, we have
  \begin{align*}
   \| f-P_t(f)\|_{L^1(\mu)}\leq \sqrt{2t}\| |\nabla f|\|_{L^1(\mu)}.
   \end{align*}
  \end{lem}

  \begin{proof}[Proof of Theorem \ref{extthm}]Take any $k+1$ non-empty, disjoint Borel subsets $A_0,A_1,\cdots,A_k\subseteq
   M$. We may assume that $\mu(A_0)\leq \mu(A_1)\leq \cdots \leq
   \mu(A_k)$, and thus
   \begin{align*}\sum_{i=0}^{k-1}\mu(A_i)\leq 1-\frac{1}{k+1} \text{ and }
    \mu(A_i)\leq 1/2 \text{ for any }i=0,1,\cdots, k-1.
    \end{align*}We put $t:=4k(k+1)/\lambda_k(M,\mu)$. We shall prove
   that there exists $i_0$, $0\leq i_0\leq k-1$, such that
   \begin{align}\label{saigo}
    \mu^+(A_{i_0})\geq (80k^3)^{-1}\sqrt{\lambda_k(M,\mu)}\mu(A_{i_0}).
    \end{align}
   For each
   $i=0,1,\cdots,k-1$, let
   $1_{A_i,\e}(x):=\min\{0,1-\frac{1}{\e}\dist(x,A_i)\}$ denote a
   Lipschitz approximation of $1_{A_i}$. Note that
   \begin{align*}
    \frac{\mu(C_{\e}(A_i))-\mu(A_i)}{\e}\geq \int_M |\nabla 1_{A_{i},\e}|d\mu,
    \end{align*}where for a Lipschitz function $f:M\to \mathbb{R}$ and $x\in M$, we put
\begin{align*}
 |\nabla f|(x):=\limsup_{y\to x} \frac{|f(y)-f(x)|}{\dist_M(y,x)}.
 \end{align*}Letting $\e\to 0$, by Lemma \ref{ledlem} we have
   \begin{align*}
    \sqrt{2t}\mu^+(A_i)\geq \| 1_{A_i}-P_t(1_{A_i})\|_{L^1(\mu)}.
    \end{align*}Since the right-side of the above inequality can be
   written as
   \begin{align*}& \int_{A_i}(1-P_t(1_{A_i}))d\mu+\int_{M\setminus
    A_{i}}P_t(1_{A_i})d\mu\\ =\ &2\Big(\mu(A_i)-\int_{A_i}P_t(1_{A_i})d\mu\Big)\\
    = \ &2 \Big(\mu(A_i)(1-\mu(A_i))-
    \int_M(P_t(1_{A_i})-\mu(A_i))(1_{A_i}-\mu(A_i))d\mu \Big),
    \end{align*}we obtain
   \begin{align}\label{s4s1}
    & \sqrt{2t}\mu^+(A_i)\\ \geq \ & 2\Big(\mu(A_i)(1-\mu(A_i))-
    \int_M(P_t(1_{A_i})-\mu(A_i))(1_{A_i}-\mu(A_i))d\mu \Big). \tag*{}
    \end{align}
Observe that $P_t(1_{A_i})-\mu(A_i)$, $i=0,1,\cdots,k-1$, are linearly
   independent and orthogonal to constant functions on $M$. Thus the
   Rayleigh quotient representation of $\lambda_k(M,\mu)$ yields that
   there exist $a_0,a_1,\cdots,a_{k-1}\in \mathbb{R}$ such that
   \begin{align}\label{s4s2}
    \lambda_k(M,\mu)\leq
    \frac{\||\nabla(\sum_{i=0}^{k-1}a_i(P_t(1_{A_i})-\mu(A_i)))|\|_{L^2(\mu)}^2}{\|\sum_{i=0}^{k-1}a_i(P_t(1_{A_i})-\mu(A_i))\|_{L^2(\mu)}^2}.
    \end{align}
   Put $f_0:= \sum_{i=0}^{k-1}a_i 1_{A_i}$. We consider the following two cases: (I) $\|f_0-\int_M f_0 d\mu
   \|_{L^2(\mu)}\geq 2\| f_0-P_t (f_0)\|_{L^2(\mu)}$, (II) $ \|f_0-\int_M f_0 d\mu
   \|_{L^2(\mu)}\leq 2\| f_0-P_t (f_0)\|_{L^2(\mu)}$.

   We prove that the case (I) cannot happen from the our choice of
   $t$. Suppose that (I) holds. In this case we get
   \begin{align}\label{s4s3}
   \|\sum_{i=0}^{k-1}a_i(P_t(1_{A_i})-\mu(A_i))\|_{L^2(\mu)}=\ & \Big\|
    P_t(f_0)-\int_M f_0 d\mu  \Big\|_{L^2(\mu)}\\ \geq \ &
    \frac{1}{2}\Big\|f_0 - \int_Mf_0 d\mu\Big\|_{L^2(\mu)}. \tag*{}
    \end{align}We estimate the right-side of the above inequality
   from below:
   \begin{claim}\label{4claim1}We have
    \begin{align*}
     \int_M \Big(\sum_{i=0}^{k-1} a_i (1_{A_i}-\mu(A_i))\Big)^2d\mu \geq
     \frac{1}{k+1}\sum_{i=0}^{k-1} a_i^2 \int_M (1_{A_i}-\mu(A_i))^2 d\mu.
     \end{align*}
    \begin{proof}Since
     \begin{align}\label{4claims1}
      \int_M (1_{A_i}-\mu(A_i))^2d\mu=\mu(A_i)(1-\mu(A_i))
      \end{align}and 
     \begin{align*}
      \int_M \Big(\sum_{i=0}^{k-1} a_i (1_{A_i}-\mu(A_i))\Big)^2d\mu=
      \sum_{i=0}^{k-1}a_i^2\mu(A_i)-\Big(\sum_{i=0}^{k-1}a_i \mu(A_i)\Big)^2,
      \end{align*}it suffices to prove
     \begin{align}\label{4claims2}
      \Big(1-\frac{1}{k+1} \Big)\sum_{i=0}^{k-1}a_i^2
      \mu(A_i)+\frac{1}{k+1}\sum_{i=0}^{k-1}a_i^2\mu(A_i)^2\geq \Big(\sum_{i=0}^{k-1}a_i \mu
      (A_i)\Big)^2.
      \end{align}
     Since
     \begin{align*}\Big( \sum_{i=0}^{k-1}a_i \mu(A_i)\Big)^2 =\ &
      \Big(\sum_{j
      =0}^{k-1}\mu(A_j)\Big)^2\cdot \Big(\sum_{i=0}^{k-1}
      \frac{\mu(A_i)}{\sum_{j=0}^{k-1}\mu(A_j)}a_i \Big)^2 \\
      \leq \ & \Big(\sum_{j=0}^{k-1} \mu(A_j) \Big)^2\sum_{i=0}^{k-1}
      \frac{\mu(A_i)}{\sum_{j=0}^{k-1}\mu(A_j)}a_i^2\\
      \leq \ & \Big(1-\frac{1}{k+1}\Big)\sum_{i=0}^{k-1}a_i^2 \mu(A_i),
      \end{align*}we have (\ref{4claims2}). This completes the proof of the claim.
     \end{proof}
    \end{claim}Claim \ref{4claim1} together with (\ref{s4s2}) and (\ref{s4s3}) implies
   the existence of $i_0$, $0\leq i_0\leq k-1$, such that
   \begin{align*}
    \lambda_k(M,\mu)\| 1_{A_{i_0}}-\mu(A_{i_0})\|_{L^2(\mu)}^2\leq
    4k(k+1) \||\nabla P_t(1_{A_{i_0}})|\|_{L^2(\mu)}^2.
    \end{align*}Using Corollary \ref{bakledcor} and $t=4k(k+1)/\lambda_k(M,\mu)$ we obtain
   \begin{align*}
    \lambda_k(M,\mu)\|1_{A_{i_0}-\mu(A_{i_0})}\|_{L^2(\mu)}^2\leq \ &
    \frac{2k(k+1)}{t}\| 1_{A_{i_0}}-\mu(A_{i_0})\|_{L^2(\mu)}^2\\
    = \ & 2^{-1} \lambda_k(M,\mu)\|1_{A_{i_0}}-\mu(A_{i_0}) \|_{L^2(\mu)}^2,
   \end{align*}which is a contradiction.

   Since (II) holds, Lemma \ref{ledlem} yields
   \begin{align}\label{s4s5}
    \frac{1}{4}\Big\| f_0-\int_M f_0
    d\mu\Big\|_{L^2(\mu)}^2\leq \ &\| P_t(f_0)-f_0 \|_{L^2(\mu)}^2\\
    \leq \ & k \sum_{i=0}^{k-1}a_i^2 \|
    P_t(1_{A_i})-1_{A_i}\|_{L^2(\mu)}^2   \tag*{}\\
    \leq \ & k \sum_{i=0}^{k-1}a_i^2 \|
    P_t(1_{A_i})-1_{A_i}\|_{L^1(\mu)} \tag*{}\\
    \leq\ & k\sqrt{2t}\sum_{i=0}^{k-1}a_i^2 \| |\nabla 1_{A_i}|\|_{L^1(\mu)} \tag*{} \\
    = \ & k\sqrt{2t} \sum_{i=0}^{k-1}a_i^2\mu^{+}(A_i). \tag*{}
    \end{align}According to Claim \ref{4claim1} and (\ref{s4s5}), there exists
   $i_0$, $0\leq i_0\leq k-1$, such that
   \begin{align*}
    \| 1_{A_{i_0}}-\mu(A_{i_0})\|_{L^2(\mu)}^2\leq 4k(k+1)\sqrt{2t}\mu^+(A_{i_0}).
    \end{align*}Thus we get
   \begin{align*}
     \int_M
    (P_t(1_{A_{i_0}})-\mu(A_{i_0}))(1_{A_{i_0}}-\mu(A_{i_0}))d\mu \leq
    \ & \|1_{A_{i_0}}-\mu(A_{i_0})\|_{L^2(\mu)}^2\\ \leq \ &
    4k(k+1)\sqrt{2t}\mu^+(A_{i_0}).
    \end{align*}Since $\mu(A_{i_0})\leq 1/2$, it follows from (\ref{s4s1}) that
   \begin{align*}
    (8k^2+8k+1)\sqrt{2t}\mu^+(A_{i_0})\geq
    2\mu(A_{i_0})(1-\mu(A_{i_0}))\geq \mu(A_{i_0}).
    \end{align*}Recalling that $t=4k(k+1)/\lambda_k(M,\mu)$, we
   finally obtain
   \begin{align*}
    \mu^+(A_{i_0})\geq
    \frac{\sqrt{\lambda_k(M,\mu)}}{(16k(k+1)+2)\sqrt{2k(k+1)}}\mu(A_{i_0})\geq \frac{\sqrt{\lambda_k(M,\mu)}}{80k^3}\mu(A_{i_0}),
    \end{align*}which implies (\ref{saigo}). This completes the proof of the theorem.
   \end{proof}

\begin{rem}\upshape From the proof of \cite{bakled} Bakry-Ledoux's lemma (Lemma \ref{bakledlem}) follows from
 the following Bakry-\'Emery type $L^2$-gradient estimate:
 \begin{align}\label{22grad}
  |\nabla P_t (f)|^2(x)\leq e^{-2 K t}P_t(|\nabla f|^2)(x)
  \end{align}for any Lipschitz function $f$ and any $x\in X$. Gigli,
 Kuwada, and Ohta proved the gradient estimate (\ref{22grad}) for compact finite-dimensional
 Alexandrov spaces satisfying CD($K,\infty$) (\cite[Theorem 4.3]{gko}). Here
 Alexandrov spaces are metric spaces whose 'sectional curvature' is bounded
     from below in the sense of the triangle comparison property. 
In particular the same argument in this section implies that Theorem \ref{extthm} holds for compact finite-dimensional Alexandrov spaces
 satisfying CD($0,\infty$). Refer to \cite{Kuwae} for
 the Laplacian on Alexandrov spaces. We remark that Theorem \ref{leethm}
 holds for compact finite-dimensional Alexandrov spaces from the proof
 of \cite{lgt}. Consequently the $k$-th eigenvalue of Laplacian and the $k$-way isoperimetric constant are equivalent up to polynomials
 of $k$ for compact finite-dimensional Alexandrov spaces satisfying
 CD$(0,\infty)$. In particular it is also valid for compact finite-dimensional Alexandrov spaces of nonnegative curvature, since such spaces
 satisfy CD$(0,\infty)$ (\cite{pet}, \cite{zhang-zhu}). 
 \end{rem}
\section{Rough stability of eigenvalues of the weighted Laplacian and multi-way
 isoperimetric constants}\label{stabilitysection}

We first review the concentration topology. Recall that the \emph{Hausdorff distance} between two closed
subsets $A$ and $B$ in a metric space $X$ is defined by
\begin{align*}
  \dist_H (A,B):= \inf \{\;\varepsilon >0 \mid A \subseteq C_{\varepsilon}(B),
  \ B \subseteq C_{\varepsilon}(A)\;\}.
\end{align*}

Let $(I,\mu)$ be a probability space.
We denote by
$\mathcal{F}(I, \mathbb{R})$ the space of all $\mu$-measurable
functions on $I$. Given 
$\lambda \geq 0$ and $f, g \in \mathcal{F}(I, \mathbb{R})$, we put 
\begin{align*}
  \me_{\lambda}(f,g):= \inf \{\;\varepsilon >0 \mid
  \mu(|f-g|> \varepsilon)\leq \lambda \varepsilon\;\},
\end{align*}where $\mu(|f-g|>\e):=\mu( \{ x\in I \mid |f(x)-g(x)|>\e
  \})$. Note that, if any two functions $f,g \in \mathcal{F}(I, \mathbb{R})$
with $f = g$ a.e. are identified to each other,
then $\me_{\lambda}$ is a distance function on $\mathcal{F}(I,\mathbb{R})$
for any $\lambda \geq 0$ and its topology on
$\mathcal{F}(I, \mathbb{R})$ coincides with the topology of the 
convergence in measure for any $\lambda >0$. The distance functions
$\me_{\lambda}$ for all $\lambda >0$ are mutually equivalent. 

Let $\dist$ be a \emph{semi-distance function on $I$},
i.e., a nonnegative symmetric function on $I\times I$
satisfying the triangle inequality.
We indicate by $\lip_1 (\dist)$ the space of
all $1$-Lipschitz functions on $I$ with respect to $\dist$.
Note that $\lip_1 (\dist)$ is a closed subset in
$(\mathcal{F}(I,\mathbb{R}),\me_{\lambda})$ for any $\lambda\geq 0$.
For $\lambda \geq 0$ and two semi-distance functions
$\dist$ and $\dist'$ on $I$, we define
  \begin{align*}
   \bobd (\dist ,\dist'):= \dist_{H}( \lip_1(\dist ), \lip_1 (\dist') ),
   \end{align*}
where $\dist_H$ is the Hausdorff distance
function in $(\mathcal{F}(X,\mathbb{R}),\me_{\lambda})$.
$\bobd$ is a distance function on the space of all
semi-distance functions on
$X$ for all $\lambda \geq 0$, and the two distance functions $\bobd$ and
$H_{\lambda'}\mathcal{L}\iota_{1}$ are equivalent to each other for any
$\lambda, \lambda' >0$.
We denote by $\mathcal{L}$ the Lebesgue measure on $\mathbb{R}$.

For any mm-space $X$ there exists a Borel measurable map
$\varphi :[\,0,1\,)\to X$ with $\varphi_{\ast}\mathcal{L}=\mu_X$
(see \cite[Theorem 17.41]{kechris}).
We call such a map $\varphi$ a \emph{parameter of $X$}.
Note that a parameter of $X$ is not unique in general.
For a parameter $\varphi$ of $X$,
we define a function $\varphi^{\ast}\dist_X:[\,0,1\,)\times [\,0,1\,)\to \mathbb{R}$
by $\varphi^{\ast}\dist_X(s,t):=\dist_X(\varphi(s),\varphi(t))$
for any $s,t \in [\,0,1\,)$. 

\begin{dfn}[Observable distance function]\label{pgd1}
  For two mm-spaces $X$ and $Y$ we define
  \begin{align*}
    \obd (X,Y):= \inf \bobd (\varphi_{X}^{\ast}\dist_X, \varphi_{Y}^{\ast}\dist_Y),
  \end{align*}
  where the infimum is taken over all parameters
  $\varphi_X :[\,0,1\,)\to X$ and $\varphi_{Y}:[\,0,1\,)\to Y$. 
\end{dfn}

We say that two mm-spaces are \emph{isomorphic} to each other if
there is a measure preserving isometry between the spaces.
Denote by $\mathcal{X}$ the space of isomorphic classes
of mm-spaces. The function $\obd$ is a distance function on $\mathcal{X}$
for any $\lambda\geq 0$. Note that $\obd$ and
$\underline{H}_{\lambda'}\mathcal{L}\iota_1$ are equivalent to each
other for any $\lambda, \lambda' > 0$.

\begin{dfn}[Concentration topology]\upshape We say that a sequence of mm-spaces $X_n$, $n=1,2,\cdots, $
 \emph{concentrates to} an mm-space $Y$ if $X_n$ converges to $Y$ as
 $n\to \infty$ with
 respect to $\obdd$. The topology on the set $\mathcal{X}$ induced by the
 observable distance function is called the \emph{concentration topology}.
 \end{dfn}

 The term 'concentration topology' comes from the following: We say that a sequence of mm-spaces $\{X_n\}$ is a \emph{L\'evy family}
 if $\lim_{n\to \infty}\alpha_{X_n}(r)=0$ for any $r>0$. Due to L\'evy's
 lemma (\cite{levy}, \cite[Proposition 1.3]{ledoux}) we obtain the following:
 \begin{prop}[{\cite{gromov}}]
 A sequence $\{X_n\}_{n=1}^{\infty}$ of mm-spaces is a L\'evy family if
  and only if it concentrates to the one-point mm-space. 
  \end{prop}
For example, the sequence of $n$-dimensional unit spheres in
$\mathbb{R}^{n+1}$, $n=1,2,\cdots$, concentrates to the one-point space
by L\'evy's result(\cite{levy}).

 The concentration topology is strictly weaker than the measured Gromov-Hausdorff
 topology on the space of mm-spaces (\cite{funa}). We mention that the concentration
 topology coincides with the measured Gromov-Hausdorff topology on the
 set of mm-spaces satisfying CD$(K,N)$ for fixed $K$ and
 $N<+\infty$. In fact, the set becomes compact with respect to the
 measured Gromov-Hausdorff topology because we have the doubling
 condition with a uniform doubling constant under the condition CD$(K,N)$.

 Answering a conjecture by Fukaya in \cite{fukaya}, Cheeger and Colding proved the
 continuity of eigenvalues of Laplacian on Riemmanian manifolds with respect to the measured
 Gromov-Hausdorff topology under the condition CD$(K, N)$ for fixed
 $K,N \in \mathbb{R}$ (\cite{ch-co}). We consider an analogy of
 the above Cheeger-Colding result with respect to the concentration topology:
\begin{cor}\label{stab}There exists a universal numeric constant $c>0$ satisfying the following. Let $\{ (M_n,\mu_n)\}$ be a sequence of closed weighted Riemannian
 manifolds of nonnegative Bakry-\'Emery Ricci curvature and assume that
 the sequence concentrates to
 a closed weighted Riemannian manifold $(M_{\infty},\mu_{\infty})$. Then for any natural number
 $k$ we have 
 \begin{align}\label{stabs1}
  \limsup_{n\to \infty} \max \Big\{
  \frac{\lambda_k(M_n,\mu_n)}{\lambda_k(M_{\infty},\mu_{\infty})},
  \frac{\lambda_k(M_{\infty},\mu_{\infty})}{\lambda_k(M_n,\mu_n)}\Big\}\leq  \exp
  (c k)
  \end{align}and
 \begin{align}\label{stabs2}
   \limsup_{n\to \infty} \max\Big\{
    \frac{h_k(M_n,\mu_n)}{h_k(M_{\infty},\mu_{\infty})},
    \frac{h_k(M_{\infty},\mu_{\infty})}{h_k(M_n,\mu_n)}     \Big\}\leq
  k^3 \exp (c k).
  \end{align}
 \end{cor}

Note that dimension of $M_n$ may diverge to infinity as $n\to \infty$. 

The rest of this subsection is devoted to prove Corollary
\ref{stab}. For the proof we first recall the definition of observable
diameter introduced by Gromov in \cite{gromov}:

\begin{dfn}[Observable diameter]\label{obdidef}Let $\kappa>0$. We define the
 \emph{partial diameter}
 \begin{align*}
  \diam (\mu_X,1-\kappa)
  \end{align*}of $\mu_X$ as the infimum of $\diam A$ over all Borel
 subsets $A\subseteq X$ with $\nu(A)\geq 1-\kappa$. Define the
 \emph{observable diameter}
 \begin{align*}
  \obs_{\mathbb{R}}(X;-\kappa)
  \end{align*}of $X$ as the supremum of $\diam (f_{\ast}\mu_X,1-\kappa)$
 over all $1$-Lipschitz functions $f:X\to \mathbb{R}$.
 \end{dfn}
 The idea of the observable diameter comes from the quantum and
 statistical mechanics, i.e., we think of $\mu_X$ as a state on a
 configuration space $X$ and $f$ is interpreted as an observable. 

The next lemma expresses the relation between the observable diameter and
the separation distance. The proof of the lemma is found in
\cite[Subsection 2.2]{funa5}
 \begin{lem}[{\cite{gromov}}]\label{obsep}Let $X$ be an mm-space. For any $\kappa,\kappa'>0$ with $\kappa>\kappa'$, we have
 \begin{enumerate}
  \item $\sep(X;\kappa,\kappa)\leq \obs_{\mathbb{R}}(X;-\kappa')$,
  \item $\obs_{\mathbb{R}}(X;-2\kappa)\leq \sep(X;\kappa,\kappa)$.
  \end{enumerate}
 \end{lem}
\begin{lem}\label{stabob}Let $X,Y$ be two mm-spaces and assume that
 $\obdd(X,Y)<\e<1$. Then for any $\kappa\in (\,\e,1\,)$ we have
 \begin{align*}
  \obs_{\mathbb{R}}(Y;-\kappa)\leq \obs_{\mathbb{R}}(X;-(\kappa-\e))+2\e.
  \end{align*}
 \begin{proof}The condition $\obdd(X,Y)<\e$ implies the existence of two parameters
  $\varphi_X:[\,0,1\,)\to X$ and $\varphi_Y:[\,0,1\,)\to Y$ such that
  \begin{align*}
  \dist_H(\varphi_X^{\ast}\lip_1(X),\varphi_Y^{\ast}\lip_1(Y))<\e.
   \end{align*}Hence, for any $f\in \lip_1(Y)$, there exists $g\in
  \lip_1(X)$ such that
  \begin{align*}
   \mathcal{L}( |f \circ \varphi_Y-g\circ \varphi_X|>\e)<\e.
   \end{align*}Take a Borel subset $A\subseteq \mathbb{R}$ such that
  $g_{\ast}\mu_X(A)\geq 1-\kappa+\e$ and $\diam
  (g_{\ast}\mu_X,1-(\kappa-\e))=\diam A$. Putting
  \begin{align*}
   B:=f\circ \varphi_Y(\{ |f\circ \varphi_Y- g\circ \varphi_X|\leq \e\} \cap (g\circ \varphi_Y)^{-1}(A)),
   \end{align*}we find
  \begin{align*}
   f_{\ast}\mu_Y(B)\geq (1-\e)+(1-\kappa+\e)-1=1-\kappa.
   \end{align*}Given $s,t\in \{ |f\circ \varphi_Y- g\circ \varphi_X|\leq \e\} \cap (g\circ \varphi_Y)^{-1}(A)$ we have
  \begin{align*}&|f\circ \varphi_Y(s)-f\circ \varphi_Y(t)|\\ \leq \ &|f\circ
   \varphi_Y(s)-g\circ \varphi_X(s)|+ |g\circ \varphi_X(s)-g\circ
   \varphi_X(t)|\\ \ & + |g\circ \varphi_X(t)-f\circ \varphi_Y(t)|\\
   \leq \ & \diam A + 2\e,
   \end{align*}which implies $\diam (f_{\ast}\mu_Y,1-\kappa)\leq \diam A
  +2\e$. This completes the proof.
  \end{proof}
 \end{lem}

\begin{lem}\label{stabl}Let $(M,\mu_M)$ and $(N,\mu_N)$ be
  two closed weighted Riemannian manifolds of nonnegative Bakry-\'Emery
  Ricci curvature such that $\obdd((M,\mu_M),(N,\mu_N))<1/2$. Assume
 that two positive numbers $\e,\delta$ satisfies 
  $\obdd((M,\mu_M),(N,\mu_N))<\e<1/2$ and $\e+\delta<1/2$. Then we have
 \begin{align}\label{stabls1}
   \lambda_1(N,\mu_N)\geq \lambda_1(M,\mu_M) \Big\{
   \frac{\delta}{2\e\sqrt{\lambda_1(M,\mu_M)}-6\log(\frac{1}{4}-\frac{\e}{2}-\frac{\delta}{2})}
   \Big\}^2
  \end{align}and
  \begin{align}\label{stabls2}
h_1(N,\mu_N)\geq h_1(M,\mu_M)\cdot \frac{\delta}{\e h_1(M,\mu_M)- 6 \log (\frac{1}{4}-\frac{\e}{2}-\frac{\delta}{2})}.
   \end{align}
  \begin{proof}
  Combining (\ref{2.2sss1}), Lemmas \ref{obsep} and
  \ref{stabob} gives that for any $\kappa> \e$ we have 
  \begin{align*}
   \obs_{\mathbb{R}}((N,\mu_N);-\kappa)\leq \ &
   \obs_{\mathbb{R}}((M,\mu_M);-(\kappa-\e))+2\e \\ \leq \ & 
   \sep \Big((M,\mu_M);\frac{\kappa-\e}{2}, \frac{\kappa-\e}{2}\Big)+2\e\\
   \leq \ & \frac{6}{\sqrt{\lambda_1(M,\mu_M)}}\log \frac{2}{\kappa-\e} + 2\e. 
   \end{align*}Lemma \ref{obsep} again yields
  \begin{align*}
   \sep ((N,\mu_N);\kappa,\kappa)\leq
   \frac{6}{\sqrt{\lambda_1(M,\mu_M)}}\log \frac{2}{\kappa-\e}+2\e.
   \end{align*}As in the proof of Lemma \ref{2.2l3} (1) we obtain 
  \begin{align*}
   \alpha_{(N,\mu_N)}(r)\leq
   \e+2\exp(-6^{-1}\sqrt{\lambda_1(M,\mu_M)}(r-2\e)  )
   \end{align*}for any $r>2\e$. By subsutituting
  \begin{align*}
   r:=2\e - \frac{6\log (\frac{1}{4}-\frac{\e}{2}-\frac{\delta}{2})}{\sqrt{\lambda_1(M,\mu_M)}}
   \end{align*}we obtain $\alpha_{(N,\mu_N)}(r)\leq 2^{-1} -
   \delta$. Applying Theorem \ref{2.1t2} then implies the inequality
   (\ref{stabls1}). The proof of (\ref{stabls2}) is similar and we omit it.
  \end{proof}
  \end{lem}
  \begin{proof}[Proof of Corollary \ref{stab}]Due to Theorem
   \ref{2.1t2} we have $\sup_{n\in \mathbb{N}}\lambda_1(M_n,\mu_n)<+\infty$
   unless $\{(M_n,\mu_n)\}$ concentrates to the one point space. Since
   the condition $CD(0,\infty)$ is preserved under
  concentration topology (\cite[Theorem 1.2]{funa6}), the limit weighted manifold
   $(M_{\infty},\mu_{\infty})$ has nonnegative Bakry-\'Emery Ricci
   curvature. Combining Lemma \ref{stabl} with Theorem \ref{Mthm} we obtain the corollary.
  \end{proof}

  The proof of Corollary \ref{stab} also follows from the following
  lemma and corollary together with Theorems \ref{Mthm} and \ref{2.1t2}:
  \begin{lem}Let $X,Y$ be two mm-spaces such that
   $\obdd(X,Y)<\e<1/(k+1)$. Then for any
   $\kappa_0,\kappa_1,\cdots,\kappa_k,\kappa_0',\kappa_1',\cdots,\kappa_k'>0$
   such that $\kappa_i-(k+1)\e \geq \kappa_i'$ for any $i$, we have 
   \begin{align*}
    \sep(Y;\kappa_0,\kappa_1,\cdots,\kappa_k) \leq \sep(X;\kappa_0',\kappa_1',\cdots,\kappa_k')+2\e.
    \end{align*}
   \begin{proof}Take $k+1$ Borel subsets $A_0,A_1,\cdots,A_k\subseteq Y$
    such that $\mu_Y(A_i)\geq \kappa_i$ for any $i$ and $\min_{i\neq
    j}\dist_Y(A_i,A_j)=\sep(Y;\kappa_0,\kappa_1,\cdots,\kappa_k)$. Since
    $\obdd(X,Y)<\e$ there exist two parameters $\varphi_X:[\,0,1 \,)\to
    X$ and $\varphi_Y:[\, 0, 1\, )\to Y$ such that
    $\bbobd(\varphi_X^{\ast}\dist_X,\varphi_Y^{\ast}\dist_Y)<\e$. For
    each $i=0,1,\cdots, k$, we put $f_i(x):=\dist_Y(x,A_i)$. Since each $f_i$
    is $1$-Lipschitz, the
    condition $\bbobd
    (\varphi_X^{\ast}\dist_X,\varphi_Y^{\ast}\dist_Y)<\e $ implies the
    existence of $k+1$ $1$-Lipschitz functions $g_i:X\to \mathbb{R}$,
    $i=0,1,\cdots,k$, such that $\me_1(f_i\circ \varphi_Y,
    g_i\circ\varphi_X)<\e$. Putting
    \begin{align*}
     \tilde{I}:= \bigcap_{i=0}^k \{ |f_i \circ \varphi_Y - g_i \circ
     \varphi_X|\leq \e\}
     \end{align*}we have $\mathcal{L}(\tilde{I})\geq 1-(k+1)\e$. For
    each $i=0,1,\cdots, k$ we define $B_i\subseteq X$ as $B_i:=\varphi_X (\varphi_Y^{-1}(A_i)\cap
    \tilde{I})$. Note that $\mu_X(B_i)\geq \mathcal{L}(\varphi_Y^{-1}(A_i)\cap
    \tilde{I})\geq \kappa_i-(k+1)\e$. For any $a_i\in \varphi_Y^{-1}(A_i)\cap
    \tilde{I}, a_j \in \varphi_Y^{-1}(A_j)\cap
    \tilde{I}$, $i\neq j$, we get
    \begin{align*}
        \dist_X(\varphi_X(a_i),\varphi_X(a_j))\geq \ &
     |g_i(\varphi_X(a_i))-g_j(\varphi_X(a_j))|\\ \geq \
     &|f_i(\varphi_Y(a_i))-f_j(\varphi_Y(a_j))|-2\e\\
     \geq \ &\dist_Y(A_i,A_j)-2\e,
     \end{align*}which implies that
    \begin{align*}
     \min_{i \neq j}\dist_X(B_i,B_j)\geq \min_{i\neq
     j}\dist_Y(A_i,A_j)-2\e = \sep(Y;\kappa_0,\kappa_1,\cdots,\kappa_k)-2\e.
     \end{align*}This completes the proof.
    \end{proof}
   \end{lem}

  \begin{cor}\label{stabsepcor}Assume that a sequence $\{X_n \}$ of
   mm-spaces concentrate to an mm-space $Y$. Then we have 
   \begin{align}
    \liminf_{n\to \infty}\sep (X_n;\kappa_0',\kappa_1',\cdots ,
    \kappa_k')\geq \sep(Y;\kappa_0,\kappa_1,\cdots,\kappa_k)
    \end{align}and
   \begin{align}
    \limsup_{n\to \infty}\sep
    (X_n;\kappa_0,\kappa_1,\cdots,\kappa_k)\leq \sep(Y;\kappa_0',\kappa_1',\cdots,\kappa_k')
    \end{align}for any
   $\kappa_0,\kappa_1,\cdots,\kappa_k,\kappa_0',\kappa_1',\cdots, \kappa_k'>0$
   such that $\kappa_i>\kappa_i'$.
   \end{cor}

 \section{Questions}

 In this section we raise several questions which are concerned with
 this paper. We also discuss conjecture which was posed in \cite{funa6}. Throughout this section, unless otherwise stated, we will always assume that
 $(M,\mu)$ is a closed weighted Riemannian manifold of nonnegative
 Bakry-\'Emery Ricci curvature.
  \begin{quest}\label{quest1}Independent of $k$, is it possible to bound $\lambda_{k+1}(M,\mu)/\lambda_{k}(M,\mu)$ or $h_{k+1}(M,\mu)/h_k(M,\mu)$ from above by a universal numeric constant ? 
  \end{quest}Masato Mimura asked me about the fraction of
 $\lambda_{k+1}(M,\mu)/\lambda_k(M,\mu)$. Theorem \ref{Mthm} leads to
 the above question for eigenvalues of the weighted Laplacian.
 Due to Theorems \ref{2.2t4} and \ref{3t2}, in order to give an affirmative
answer to Question \ref{quest1} for eigenvalues it suffices to extend E.~Milman's theorem
(Theorem \ref{2.1t2}) in terms of
$\lambda_k(M,\mu)$ and the $k$-separation
distance, i.e., any $k$-separation inequalities imply appropriate lower bounds of the $k$-th
eigenvalue $\lambda_k(M,\mu)$. Or more weakly, it suffices to prove that
 any logarithmic
$k$-separation inequalities of the form (\ref{3ss1}) give appropriate
estimates of the $k$-th eigenvalue $\lambda_k(M,\mu)$ from below. This
 can also be considered as an extension of \cite[Theorem 1.14]{goz}. In
 \cite{goz} Gozlan, Roberto, and Samson proved that any exponential concentration inequalities imply appropriate Poincar\'e
     inequalities under assuming CD$(0,\infty)$. Notice that by Lemma \ref{2.2l3} exponential concentration inequalities are nothing but
     logarithmic $2$-separation inequalities.

     For multi-way isoperimetric constants, we also need to improve $k^3$
     order in Proposition \ref{leeprop} to some universal numeric constant. The
     following integration argument makes possible to improve $k^3$ order
     but it is not logarithmic separation inequalities:

 \begin{prop}Let $(M,\mu)$ be a closed weighted Riemannian
  manifold and $k$ a natural number. Then for any $\kappa>0$ we have
\begin{align*}
 \sep ((M,\mu); \underbrace{\kappa,\kappa,
 \cdots,\kappa}_{k+1\text{ times}})\leq \frac{2}{\log 2}\cdot\frac{\log (2/\kappa)}{h_k(M,\mu)\kappa}.
 \end{align*}
  \begin{proof}Let $A_0, A_1, \cdots, A_k$ be $k+1$ Borel subsets of $M$
   such that $\mu(A_i)\geq \kappa$ for any $0\leq i\leq k$. Our goal is to prove the following inequality:
   \begin{align}\label{kws1}
    D:=\min_{i\neq j}\dist_M(A_i,A_j)\leq \frac{2}{\log 2}\cdot\frac{\log (2/\kappa)}{h_k(M,\mu)\kappa}. 
    \end{align}
   In order to prove (\ref{kws1}) we may assume that each $A_i$ is given
   by a finite union of open balls. For $r\in [0,D/2)$ we put $B_i:=
   O_{r}(A_i)$, $i=0,1,\cdots, k-1$, and $B_k:=M\setminus
   \bigcup_{i=0}^{k-1}B_i$. By the definition of $h_k(M,\mu)$, we have
   $\mu^{+}(B_{i_0})\geq h_k(M,\mu)\mu(B_{i_0})$ for some $i_0$. Assume
   first that
   $i_0=k$. Since each $B_i$ consists of a finite union of open balls we obtain
 \begin{align*}
  \sum_{i=0}^{k-1}\mu^{+}(B_i)=\mu^{+}(B_k)\geq h_{k}(M,\mu)\mu(B_k)\geq h_k(M,\mu)\kappa. 
  \end{align*}In the case where $i_0\leq k-1$, we get
  \begin{align*}
     \sum_{i=0}^{k-1}\mu^{+}(B_i)\geq
  \mu^{+}(B_{i_0})\geq h_k(M,\mu)\mu (B_{i_0}) \geq h_k(M,\mu )\kappa
   \end{align*}Combining the above two inequalities implies that 
   \begin{align*}
    \mu\Big(\bigcup_{i=0}^{k-1}O_r(A_i)
    \Big)-\mu\Big(\bigcup_{i=0}^{k-1} A_i\Big) = \int_0^r
    \sum_{i=0}^{k-1}\mu^{+}(O_{s}(A_i)) ds\geq h_k(M,\mu)\kappa r,
    \end{align*}which yields
 \begin{align}\label{kws2}
  \mu\Big(M\setminus \bigcup_{i=0}^{k-1} O_r(A_i)\Big)\leq (1- h_k(M,\mu) \kappa
  r )\mu \Big(M\setminus \bigcup_{i=0}^{k-1} A_i\Big).
  \end{align}What follows is a straightforward adaption of
   Gromov-V.~Milman's argument in \cite[Theorem 4.1]{milgro}. 
   Put $\e:=(2\kappa h_k(M,\mu))^{-1}$. If $\e \leq r$, then
   there exists a natural number $j$ such that $j\e \leq r<(j+1)\e$. Iterating
   (\ref{kws2}) $k$ times shows
   \begin{align*}
    \mu\Big(M\setminus O_r \Big( \bigcup_{i=0}^{k-1}A_i\Big)\Big)\leq \ &
    \mu\Big(M\setminus O_{j \e} \Big( \bigcup_{i=0}^{k-1}A_i\Big)\Big)\\
    \leq \ &
    (1-h_k(M,\mu)\kappa \e)\mu\Big(M\setminus O_{(j-1) \e} \Big(
    \bigcup_{i=0}^{k-1}A_i\Big)\Big)\\
    \ & \cdots \\
    \leq \ & (1-h_k(M,\mu)\kappa \e)^{j}\mu\Big(M\setminus
    \bigcup_{i=0}^{k-1}A_i \Big)\\
    \leq \ & (1-h_k(M,\mu)\kappa \e)^{j}\\
    = \ & \exp(-j\log 2)\\
    \leq \ & \exp (-(r/\e)\log 2)\\
    = \ & \exp (-h_k(M,\mu)r \kappa 2\log 2 ). 
    \end{align*}
   If $r<\e$, then we have
   \begin{align*}
    \mu\Big(M\setminus O_r \Big( \bigcup_{i=0}^{k-1}A_i\Big)\Big)\leq
    1\leq 2\cdot 2^{-\e^{-1}r}\leq 2\exp(-h_k(M,\mu)r\kappa 2\log 2 )
    \end{align*}
   Put $r:=D/2$. Combining the above two inequalities we obtain
   \begin{align*}
    \kappa \leq \mu(A_k)\leq  \mu\Big(M\setminus O_{\frac{D}{2}} \Big(
    \bigcup_{i=0}^{k-1}A_i\Big)\Big)\leq 2\exp (-h_k(M,\mu)D \kappa \log 2 ),
    \end{align*}which implies (\ref{kws1}). This completes the proof. 
\end{proof}
  \end{prop}
  \begin{quest}\label{quest2}What is the right order of $\sqrt{\lambda_k(M,\mu)}/h_k(M,\mu)$, $\lambda_k(M,\mu)/\lambda_1(M,\mu)$,
   and $h_k(M,\mu)/h_1(M,\mu)$ in $k$? Especially
    can we bound $\lambda_k(M,\mu)/\lambda_1(M,\mu)$ and
    $h_k(M,\mu)/h_1(M,\mu)$ from above by some polynomial function of $k$ ?
   \end{quest}

   The following two questions are concerned with the stability of
   eigenvalues of the weighted Laplacian and multi-way isoperimetric constants.

   \begin{quest}\label{quest3}Is it true that if two convex domains $K,L\subseteq
    \mathbb{R}^n$ satisfy $\vol(K)\simeq \vol(L)$, then
     $\eta_k(K)\simeq \eta_k(L)$ or $h_k(K)\simeq h_k(L)$? 
    \end{quest}

    \begin{quest}\label{quest4}Can we get the stability of eigenvalues
     of the weighted Laplacian and
     multi-way isoperimetric constants with respect to the concentration
     topology ? Or more weakly can we replace $\exp(c k)$ and $k^3\exp
     (c k)$ in Corollary \ref{stab} with some universal numeric constant ?
     \end{quest}In view of Corollary \ref{stabsepcor} an
     extension of E.~Milman's theorem for the $k$-separation distance and
     the $k$-th eigenvalue would imply the latter question in
     Question \ref{quest4}.

 In \cite[Conjecture 6.11]{funa6} we raised the following conjecture.

 \begin{conj}\label{funaconj}For any natural
 number $k$ there exists a positive constant $C_k$ depending only on $k$
 such that if $X$ is a compact finite-dimensional Alexandrov space of nonnegative
 curvature, then we have
  \begin{align*}
   \lambda_k(X)\leq C_k \lambda_1(X).
   \end{align*}
  \end{conj}

     Since Theorems \ref{leethm} and \ref{extthm} hold for
     compact finite-dimensional Alexandrov spaces of nonnegative curvature, the above question amounts to saying the
     existence of $C_k$ such that $h_k(X)\leq C_k h_1(X)$.

     We remark that Theorem \ref{2.2t4} holds for compact finite-dimensional Alexandrov
     spaces. In fact, the only we need in the proof is the
     Davies-Gaffney heat kernel
     estimate
     \begin{align*}
      \int_A\int_B p_t(x,y)d\mu(x)d\mu(y)\leq \sqrt{\mu(A)\mu(B)}\exp\Big(-\frac{\dist^2(A,B)}{4t}\Big)
      \end{align*}for any Borel subsets $A,B$ and asymptotic expansion of heat kernel by eigenvalues and
     eigenfunctions of Laplacian (\cite{cgy1}). These are true for
     compact finite-dimensional Alexandrov spaces (\cite{sturm4},
     \cite{Kuwae}). However it is not known the corresponding theorem of E.~Milman's theorem
     (Theorem \ref{2.1t2}) for Alexandrov spaces. Note that we
     used Theorem \ref{2.1t2} in the proof of
     Theorem \ref{3t2}. In order to give an
     affirmative answer to Conjecture \ref{funaconj}, it suffices to prove that
     any concentration inequalities imply appropriate exponential
     concentration inequalities under assuming CD$(0,\infty)$ or
     Theorem \ref{3t2} holds for general CD$(0,\infty)$ spaces by
     Gozlan-Roberto-Samson's theorem \cite[Theorem 1.14]{goz}.

\bigbreak
\noindent
\bigbreak
\noindent
{\it Acknowledgments.}The author would like to thank to Professors Alexander
Bendikov, Alexander Grigor'yan, Emanuel Milman,
Kazuhiro Kuwae, Karl Theoder Sturm, 
and Nathael Gozlan for their comments and their interests of this paper. He also
thanks to Professor Masato Mimura
for several discussion. A part of this work was done while the author visited Bonn
university and Bielefeld university. 


\section{Appendix}
The only point we need to be care when we prove
Lee-Gharan-Trevisan's theorem (Theorem \ref{leethm}) for the smooth
setting is the following lemma:

\begin{lem}[{\cite[Lemma 2.1]{lgt}}]Let $X$ be an mm-space and $f:X\to \mathbb{R}^n$ a Lipschitz map. Then there
 exists a closed subset $A$ of $X$ such that $ A \subseteq \supp f$ and 
 \begin{align*}
  \frac{\mu_X^+(A)}{\mu_X(A)}\leq 2 \frac{\| |\nabla f| \|_{L^2(\mu_X)}}{\|  f \|_{L^2(\mu_X)}}.
  \end{align*}
 \begin{proof}For any positive real number $t$ we put
  \begin{align*}A_t:=\{ x\in X \mid |f(x)|^2\geq t \}.
   \end{align*}Note that $A_t\subseteq
  \supp f$ for any $t>0$ and 
  \begin{align}\label{aps1}
   \int_0^{\infty} \mu_X(A_t) dt = \|f\|_{L^2(\mu_X)}^2
  \end{align}The co-area inequality (\cite[Lemma 3.2]{bobhou}) implies that
   \begin{align}\label{aps2}
   \int_0^{\infty} \mu_X^+ (A_t) dt \leq \ &\int_M |\nabla (|f|^2)|(x)
    d\mu_X(x)\\ \leq \ &2\int_M
    |f(x)||\nabla f|(x) d\mu_X(x) \tag*{}\\
    \leq \ &2 \| f\|_{L^2(\mu_X)} \||\nabla f| \|_{L^2(\mu_X)}. \tag*{} 
   \end{align}Combining (\ref{aps1}) with (\ref{aps2}) gives
  \begin{align*}
   \frac{\int_0^{\infty}\mu_X^+ (A_t)dt}{\int_0^{\infty} \mu_X(A_t) dt }\leq 2 \frac{\| |\nabla f
   | \|_{L^2(\mu_X)}}{\|f\|_{L^2(\mu_X)}},
   \end{align*}which implies the conclusion of the lemma.
 \end{proof}
 \end{lem}
\end{document}